%% file: main.tex
\documentclass{article}
\usepackage{mathrsfs}
\usepackage[utf8]{inputenc}
\usepackage[letterpaper,margin=1.0in]{geometry}
\input{math.tex}

\usepackage{color}
\usepackage{hyperref}
\usepackage{enumerate}
\usepackage{enumitem}
\usepackage{algorithm}
\usepackage{algorithmic}
\usepackage[numbers,sort,compress]{natbib}
\usepackage{url}
\usepackage{microtype}
\usepackage{algorithm}
\usepackage{algorithmic}
\usepackage{hyperref}
\usepackage{url}
\usepackage{microtype}
\usepackage{enumitem}
\usepackage{graphicx}
\usepackage{xspace}
\usepackage{booktabs}
\usepackage{empheq}
\usepackage{bbding}
\usepackage{makecell}

\title {Optimal Asynchronous Stochastic Nonconvex Optimization \\ under Heavy-Tailed Noise}
\date{}
\author{Yidong Wu \qquad\qquad Luo Luo}

\begin{document}
\maketitle

\begin{abstract}
This paper considers the problem of asynchronous stochastic nonconvex optimization with heavy-tailed gradient noise and arbitrarily heterogeneous computation times across workers.
We propose an asynchronous 
normalized stochastic gradient descent algorithm with momentum.
The analysis show that our method achieves the optimal time complexity under the assumption of bounded $p$th-order central moment with $p\in(1,2]$.
We also provide numerical experiments to show the effectiveness of proposed method.
\end{abstract}

\section{Introduction}
This paper studies stochastic optimization problem 
\begin{align}\label{main:prob}
\min_{x\in\mathbb{R}^d}f(x):= \mathbb{E}_{\xi\sim\mathcal{D}}[F(x;\xi)],    
\end{align}
where $\xi$ is the random variable that follows distribution $\fD$ and $F(x;\xi)$ is the stochastic component which is smooth but possibly nonconvex.
We consider the distributed setting that one server and $n$ workers collaboratively solve problem~(\ref{main:prob}).
We focus on the asynchronous scenario that each worker can access distribution $\fD$ to achieve the unbiased stochastic gradient estimate for the objective with heavy-tailed noise under arbitrarily heterogeneous computation time.

In practical distributed systems, workers may perform the heterogeneous delays due to differences in hardware and network communications \citep{horvath2021fjord,kairouz2021advances,ryabinin2021moshpit,dean2013tail}.
The asynchronous stochastic first-order methods are popular to address this issue, which allows each worker to send its stochastic gradient estimate to server immediately once the local computation is finished,
and the server can perform the update based on the currently received data without waiting for other workers \citep{recht2011hogwild,tsitsiklis2003distributed,ananthanarayanan2013effective,agarwal2011distributed,feyzmahdavian2016asynchronous,shi2024ordered,lian2018asynchronous,dutta2018slow}.
This kind of strategies keeps all workers busy and their empirical performance is superior to the synchronous methods in many practical situations \cite{assran2020advances}.

For single machine scenario, the vanilla stochastic gradient descent (SGD) method \cite{robbins1951stochastic,ghadimi2013stochastic,bottou2018optimization,bottou2007tradeoffs} attains the optimal stochastic first-order oracle complexity for finding approximate stationary points of the smooth nonconvex objective under bounded-variance assumption for the stochastic gradient estimate \citep{arjevani2023lower}.
For distributed setting, the performance of asynchronous optimization stochastic algorithms additionally depends on the delays caused by straggling workers.
In seminal works, \citet{mishchenko2022asynchronous,koloskova2022sharper} investigated asynchronous SGD (ASGD) under the fixed computation model where  the time complexity required for each worker to access a single stochastic gradient estimate has a fixed but worker-dependent upper bound.
Later, \citet{tyurin2023optimal} proposed Rennala SGD 
by including an asynchronous mini-batch collection strategy \citep{dutta2018slow}, which additionally achieves the optimal time complexity for a more general universal computation model, i.e., 
the time complexity for each worker to access a single stochastic gradient is not upper bounded by the fixed values but can change arbitrarily in time \citep{tyurin2025tight}.
Recently, \citet{pmlr-v267-maranjyan25b} established Ringmaster ASGD by introducing a delay threshold for the server to discard the stale stochastic gradient estimates, also achieving the optimal time complexity to both fixed and universal computation models.
The updates of Ringmaster ASGD is performed immediately upon receiving a fresh but not too outdated stochastic gradient, which aligns with the ideas of lock-free asynchronous and achieves the better empirical performance than Rennala SGD.
Note that all above algorithms and theory depend on the assumption of bounded variance, which does not always hold in practice \cite{laurent2024linear,battash2024revisiting,zhang2020adaptive,gurbuzbalaban2021,garg2021proximal,cayci2023provably,zhu2024robust,simsekli2019}. 

In practice, heavy-tailed gradient noise is widely observed in empirical studies of popular applications such as the language model training 
\cite{laurent2024linear,battash2024revisiting,zhang2020adaptive,gurbuzbalaban2021} and reinforcement learning \cite{garg2021proximal,cayci2023provably,zhu2024robust,simsekli2019}, going beyond the standard bounded-variance assumption.
In this regime, the vanilla SGD may fail to converge because it is easily influenced by the large gradient noise \cite{zhang2020adaptive}.
The algorithms including clipped and normalized SGD and their variants are broadly
used to deal with heavy-tailed gradient noise in training process \citep{gorbunov2020stochastic,gorbunov2022clipped,sadiev2023high,gorbunov2024high,zhang2020adaptive,nguyen2023high,cutkosky2020momentum,liu2023stochastic,hubler2025,liunonconvex2025,sun2024gradient,he2025complexity,liu2025stochastic}.
Specifically, tight complexity bounds for stochastic first-order optimization under heavy-tailed noise have been established for the single machine scenario under the assumption of finite $p$th central moment with $p\in(1,2]$ \cite{liu2023stochastic,hubler2025,sadiev2023}. 
Several recent works studied distributed stochastic optimization with heavy-tailed noise \cite{yu2025decentralized,sun2025distributed,zhang2025federated,wang2025near}, while their results are only restricted to the synchronous case. 

In this paper, we establish the first optimal asynchronous stochastic nonconvex optimization under heavy-tailed noise with the assumption of bounded $p$th central moment.
We summarize  main contributions of our work as follows.
\begin{itemize}[leftmargin=0.5cm,topsep=-0.15cm,itemsep=-0.1cm]
\item We design a novel asynchronous stochastic first-order  method by incorporating the momentum and the step of normalization into updates on server.
We also discard the stale stochastic gradients to address the possibly straggling workers.
\item We prove that our method can find an $\epsilon$-stationary point with the time complexity of
\begin{equation*}
\fO\left(
\min_{m\in[n]}
\left(\frac{1}{m}\sum_{i=1}^m \frac{1}{\tau_i}\right)^{-1}
\left(
\frac{L\Delta}{\epsilon^{2}}
+
\frac{L\sigma^{\frac{p}{p-1}}\Delta}{m \epsilon^{\frac{3p-2}{p-1}}}
\right)\right)
\end{equation*}
for the fixed computation model,
where $n$ is the number of workers, $\tau_i>0$ is the time to compute a single stochastic gradient at the $i$th worker, $L>0$ is the smoothness parameter, $\Delta>0$ is the initial function value gap, $\sigma>0$ is the noise level, and $p\in(1,2]$ is the order of moment.
\item We then provide the lower bound on the time complexity to show the above upper bounds is tight for the fixed computation model.
\item We further extend our results to the setting of general computational dynamics, also achieving the optimal time complexity bound.
\end{itemize}

\section{Preliminaries}
\label{sec:pre}

In this section, we formalize the assumptions on problem~(\ref{main:prob}), then introduce the computational models considered for asynchronous optimization algorithms.

We impose the following assumptions for the objective.
\begin{asm}\label{ass:smooth}
We suppose the objective $f: \mathbb{R}^d \to \mathbb{R}$ is $L$-smooth and lower bounded, i.e., there exists $L>0$ such that
$\| \nabla f(x) - \nabla f(y) \| \le L \| x - y \|$
for all $x, y \in \mathbb{R}^d$
and it holds $f^\star:=\inf_{x\in\BR^d}f(x)>-\infty$.
\end{asm}

We suppose each worker can access the stochastic first-order oracle with $p$th bounded central moment ($p$-BCM) \citep{zhang2020adaptive}.

\begin{asm}[\citet{zhang2020adaptive}]\label{ass:pBCM}
We suppose each worker can draw $\xi\sim\fD$ to achieve an unbiased stochastic gradient estimate $\nabla F(x;\xi)$ with  the $p$th bounded central moment for given $x\in\BR^d$, i.e., there exists $p\in(1,2]$ and $\sigma\geq 0$ such that it holds
$\mathbb{E} [ \nabla F(x; \xi)] = \nabla f(x)$ and 
$\mathbb{E}\left[\|\nabla F(x; \xi) - \nabla f(x) \|^p\right] \le \sigma^p$ for all $x\in\BR^d$.
\end{asm}

We consider the following fixed computation model for asynchronous stochastic optimization \cite{mishchenko2022asynchronous}.

\begin{asm}[{\citet{mishchenko2022asynchronous}}]\label{ass:time}
We suppose each worker $i\in[n]$ can take no more than~$\tau_i>0$ units of wall-clock time to draw $\xi\sim\fD$ and  access a single stochastic gradient estimate $\nabla F(x;\xi)$ for given $x\in\BR^d$.
\end{asm}

We are also interested in the universal computation model (UCM) as follows \citep{tyurin2025tight}.

\begin{asm}\label{ass:universalmodel}
For each worker $i\in[n]$, the number of stochastic gradients can be calculated from a time $t_0$ to a time $t_1$ is the Riemann integral of the computation power followed by the floor operation, i.e., 
\begin{align*}
\left\lfloor \int_{t_0}^{t_1} \upsilon_i(\tau)\,{\rm d}\tau \right\rfloor=\lfloor V_i(t_1)-V_i(t_0)\rfloor,
\end{align*}
where $\upsilon_i:\BR_+\to\BR_+$ is the computation power of worker $i$ which is continuous almost everywhere
and
\begin{align*}
 V_i(t):=\int_{0}^{t} \upsilon_i(\tau)\,{\rm d}\tau
 \end{align*}
defined on $t\in\R_{\ge 0}$ is the cumulative computation work function of worker $i$.
\end{asm}

\begin{remark}    
The integral of function $\upsilon_i$ over a given interval describes the computation work performed on worker $i$. 
Intuitively, the small $\upsilon_i$ indicates worker~$i$ performs
the less computation. 
Conversely, the large $\upsilon_i$ indicates worker~$i$ can perform the more computation.
In addition, the universal computation model (Assumption \ref{ass:universalmodel}) reduces to the fixed computation model (Assumption \ref{ass:time}) when
$\upsilon_i(t) = 1/\tau_i$ for all $t\geq 0$ and~$i\in[n]$ \citep{tyurin2025tight}. 
\end{remark}

\section{The Algorithm and Main Results}\label{sec:alg-main}

We proposed our Ringmaster Asynchronous Normalized Stochastic Gradient Descent with Momentum (RANSGDm) in Algorithm \ref{alg:ringmaster_nsgd_mom_server}.
The main steps of proposed algorithm on the server can be summarized as
\begin{align}\label{eq:main-steps}    
\begin{cases}
v_{k+1} \gets \beta_k v_k + (1-\beta) \nabla F(x_{k-\delta_k};\xi_{k-\delta_k}), \\[0.15cm]
x_{k+1} \gets x_k - \dfrac{\eta v_{k+1}}{\|v_{k+1}\|},
\end{cases}
\end{align}
which only performs when the delay $\delta_k$ is smaller than the threshold $R$. 
Here, the notation $\xi_{k-\delta_k}$ correspond to the sample $\xi_{k,i}$ achieved from line 23 of Algorithm \ref{alg:ringmaster_nsgd_mom_server} for establishing the gradient estimator at point $x_{k-\delta_k}$.
The key differences between RANSGDm and existing Ringmaster ASGD \citep{pmlr-v267-maranjyan25b} include that 
the update on $x_{k+1}$ contains the normalization and the direction $v_{k+1}$ involves a momentum correction to the stochastic gradient with the momentum parameter $\beta_k$ such that $\beta_k=\beta\in[0,1)$ when $k>1$ and $\beta_k=0$ otherwise.
It is worth noting that introducing  normalization and momentum is necessary, since heavy-tailed noise leads to simply iterating along with the stochastic gradient may fail to convergence even for the case of single machine optimization \cite{zhang2020adaptive}.
In addition, the delay threshold $R$ mitigates the effects of old stochastic gradient with heavy-tailed noise, which is also well compatible with techniques of normalization and momentum to make the algorithm be efficient.

We present the main theoretical result for our RANSGDm (Algorithm \ref{alg:ringmaster_nsgd_mom_server}) in the following theorem.

\begin{thm} 
\label{thm:main_mom_tight}
Under Assumptions \ref{ass:smooth} and \ref{ass:pBCM},
we run our RANSGDm (Algorithm \ref{alg:ringmaster_nsgd_mom_server}) by taking
\begin{align*}
\begin{split}    
R=\left\lceil \frac{1}{\alpha}\right\rceil,~\eta=\frac{\alpha\epsilon}{24L},~\beta=1-\alpha,~K=\left\lceil \frac{72L\Delta}{\alpha\epsilon^2}\right\rceil\!,
\end{split}
\end{align*}
where 
\begin{equation*}
\alpha
=\min\left\{1,\left(\frac{\epsilon}{3\cdot 2^{\frac{p+1}{p}}\sigma}\right)^{\frac{p}{p-1}}\right\},
\quad\Delta=f(x_0) - f^\star.
\end{equation*}
Then the output $\hat x\in\BR^d$ holds $\E\|\nabla f(\hat x)\|\le \epsilon$.
\end{thm}

Theorem \ref{thm:main_mom_tight} indicates that our RANSGDm can achieve an $\epsilon$-stationary point by performing 
\begin{align}\label{eq:main-K}
K=\fO\left(\frac{L\Delta}{\epsilon^2} + \frac{L\sigma^{\frac{p}{p-1}}\Delta}{\epsilon^{\frac{3p-2}{p-1}}}\right)    
\end{align}
rounds of updates on the server, where each round of update  corresponds to one execution of lines 10--13 in Algorithm~\ref{alg:ringmaster_nsgd_mom_server}.
The order of such $K$ aligns with the optimal first-order
oracle complexity of the synchronous stochastic nonconvex optimization in non-distributed setting  \cite{zhang2020adaptive,liu2023stochastic}.

For the asynchronous scenario, the number of update rounds on the server does not correspond to the wall-clock time. 
For the fixed computation model (Assumption~\ref{ass:time}),
the setting of delay threshold $R$ guarantees any $R$ consecutive rounds of updates on the worker takes at most  
\begin{equation}\label{eq:tR-main-fixed}
t(R) := 2 \min_{m\in[n]} \left( \frac{1}{m}\sum_{i=1}^m \frac{1}{\tau_i} \right)^{-1} \left( 1 + \frac{R}{m} \right)
\end{equation}
units of time \citep[Lemma 4.1]{pmlr-v267-maranjyan25b}, where~$\tau_i$ is the time complexity for achieving an stochastic gradient estimate at worker $i\in[n]$.
Hence, the algorithm requires the overall time complexity of $t(R)\cdot\lceil K/R \rceil$ for finding an $\epsilon$-stationary point.
Following results of Theorem~\ref{thm:main_mom_tight},
we combine equations (\ref{eq:main-K}) and (\ref{eq:tR-main-fixed}) to achieve the complexity for fixed computation model as follows.

\begin{algorithm}[tb]
\caption{Ringmaster-ANSGD-m}
\label{alg:ringmaster_nsgd_mom_server}
\begin{algorithmic}[1]
\STATE \textbf{Input:} threshold $R\in\mathbb{N}$, stepsize $\eta>0$, momentum parameter $\beta\in[0,1)$, and  iteration number $K\in\BN$.  
\STATE \textbf{\underline{server:}} \\[0.02cm]
\STATE\quad $x_0\in\mathbb{R}^d$,~~$v_0 \gets 0$,~~$k\gets 0$ \\[0.02cm]
\STATE\quad \textbf{while} $k<K$ \textbf{do} \\[0.02cm]
\STATE\quad\quad \textbf{if} receive a data request \textbf{then} \\[0.02cm]
\STATE\quad\quad\quad send $(x_k, k)$ \\[0.02cm]
\STATE\quad\quad \textbf{else if} receive $(g_k,\iota_k)$  \textbf{then} \\[0.02cm]
\STATE\quad\quad\quad $\delta_k \gets k - \iota_k$ \\[0.02cm]
\STATE\quad\quad\quad \textbf{if} $\delta_k < R$ \textbf{then} \\[0.02cm]
\STATE\quad\quad\quad\quad $\beta_k \gets \begin{cases}
          0,  & k\leq 1 \\[0.02cm]
          \beta,  & k > 1
          \end{cases}$ \\[0.02cm]
\STATE\quad\quad\quad\quad $v_{k+1} \gets \beta_k v_k + (1-\beta) g_k$ \\[0.02cm]
\STATE\quad\quad\quad\quad $x_{k+1} \gets x_k -  \dfrac{\eta v_{k+1}}{\|v_{k+1}\|}$ \\[0.02cm]
\STATE\quad\quad\quad\quad $k \gets k+1$ \\[0.02cm]
\STATE\quad\quad\quad \textbf{end if} \\[0.02cm]
\STATE\quad\quad \textbf{end if} \\[0.02cm]
\STATE\quad \textbf{end while} \\[0.02cm]
\STATE\quad \textbf{Output:} $\hat x\sim{\rm Unif}(\{x_0,\dots,x_{K-1}\})$ \\[0.03cm]
\STATE\underline{\textbf{worker} $i\in[n]$\textbf{:}} \\[0.02cm]
\STATE\quad \textbf{while} true \textbf{do}
\STATE\quad\quad send a data request  \\[0.02cm]
\STATE\quad\quad receive data $(x_k, k)$  \\[0.02cm]
\STATE\quad\quad $\xi_{k,i}\sim\fD$ \\[0.02cm]
\STATE\quad\quad $g_{k,i} \gets \nabla F(x_{k};\xi_{k,i})$ \\[0.02cm]
\STATE\quad\quad send $(g_{k,i},k)$  \\[0.02cm]
\STATE\quad \textbf{end while}
\end{algorithmic}
\end{algorithm}

\begin{cor}
\label{cor:time_mom_refined}
Under Assumptions \ref{ass:smooth}, \ref{ass:pBCM}, and \ref{ass:time}, 
we run Algorithm \ref{alg:ringmaster_nsgd_mom_server} by following the setting of Theorem \ref{thm:main_mom_tight},
then it achieves an $\epsilon$-stationary point with the time complexity of
\begin{equation*}
\fO\left(
\min_{m\in[n]}
\left[
\left(\frac{1}{m}\sum_{i=1}^m \frac{1}{\tau_i}\right)^{-1}
\left(
\frac{L\Delta}{\epsilon^{2}}
+
\frac{L\sigma^{\frac{p}{p-1}}\Delta}{m \epsilon^{\frac{3p-2}{p-1}}}
\right)
\right]\right).
\end{equation*}
\end{cor}

For the universal computation model (Assumption \ref{ass:universalmodel}),  
any~$R$ consecutive rounds of updates from time $T_0$ on the worker takes at most
\begin{align*}
T(R,T_0):=
\min\left\{
T \ge 0: \frac{1}{4}\sum_{i=1}^{n}\int_{T_0}^{T} \upsilon_i(\tau)\,{\rm d}\tau \ge R
\right\}
\end{align*}
units of time \citep[Lemma 5.1]{pmlr-v267-maranjyan25b}.
We then combine the form of $T(R,T_0)$ with the setting of $K$ in equation (\ref{eq:main-K}) 
to achieve the time complexity for the universal computation model as follows.

\begin{cor}
\label{lem:universal_time_recursion_blocks}
Under Assumption \ref{ass:universalmodel}, we run Algorithm \ref{alg:ringmaster_nsgd_mom_server} by following the setting of Theorem \ref{thm:main_mom_tight},
then it can achieve an $\epsilon$-stationary point after at most $T_{\bar K}$ seconds, where 
\begin{align*}
\bar K := \left\lceil \frac{74L\Delta}{\epsilon^2} \right\rceil=\fO\left(
\frac{L\Delta}{\epsilon^{2}}\right),    
\end{align*}
and $T_{\bar K}$ is the ${\bar K}$th element of the following recursively defined sequence $\{T_K\}_{K\geq 1}$ such that $T_0=0$ and 
\begin{align*}
T_K := \min\left\{T\ge 0: \frac{1}{4}\sum_{i=1}^n \int_{T_{K-1}}^{T} \upsilon_i(\tau)\,{\rm d} \tau\ge R\right\}
\end{align*}
for all $K\ge 1$, 
where $R=\left\lceil 1/\alpha\right\rceil=\Theta(\lceil\sigma/\epsilon\rceil^{\frac{p}{p-1}})$.
\end{cor}

\begin{remark}
In the case of $p=2$, the time complexity bounds achieved by our Corollaries \ref{cor:time_mom_refined} and \ref{lem:universal_time_recursion_blocks} match the optimal results for 
fixed and universal computation models under the bounded variance setting, respectively \citep[Theorems 4.2 and 5.1]{pmlr-v267-maranjyan25b}.
\end{remark}

\section{The Complexity Analysis}

In this section, we briefly sketch the proof of our main result in Theorem \ref{thm:main_mom_tight}, and the details are deffered to Appendix~\ref{sec:ringmaster_nsgd_momentum}.
Recall that update (\ref{eq:main-steps}) on the server depends on the stale  gradient estimate $\nabla F(x_{k-\delta_k};\xi_{k-\delta_k})$, so that we denote 
\begin{align*}
    z_k = x_{k-\delta_k}
\end{align*}
in the remainder of this paper.  

The analysis starts from the observation that is specific to our algorithm design, 
i.e., \emph{every round of update on the server moves by a fixed distance $\eta$ due to the step of normalization}.
Together with the delay threshold $R$, it immediately yields an upper bound on the staleness. Specifically, we have the following lemma.
\begin{lem}
\label{lem:staleness}
Based on Algorithm \ref{alg:ringmaster_nsgd_mom_server}, the variables on the worker holds 
\begin{align*}
\|x_k-z_k\| = \|x_k - x_{k-\delta_k} \| \le R\eta.
\end{align*}
for all $k\geq 1$.
\end{lem}

Based on the smoothness condition in Assumption~\ref{ass:smooth}, 
we describe the descent for the objective function value as follows.

\begin{lem}
\label{lem:descent_mom}
Under Assumptions~\ref{ass:smooth},
Algorithm \ref{alg:ringmaster_nsgd_mom_server}
holds
\begin{align*}
\E\big[f(x_{k+1})\big]
\le
\E\big[f(x_k)\big]
-\eta\E\|\nabla f(x_k)\|+2\eta\E\|v_{k+1}-\nabla f(x_k)\|
+\frac{L\eta^2}{2}
\end{align*}
for all $k\geq 0$.
\end{lem}

We now focus on the term $\E\|v_{k+1}-\nabla f(x_k)\|$ in the upper bound of 
Lemma~\ref{lem:descent_mom}.
The main challenge for the analysis is that the direction 
\begin{align*}
    v_{k+1} \gets \beta_k v_k + (1-\beta) \nabla F(z_k;\xi_{k-\delta_k})
\end{align*}
is established by the heavy-tailed noise stochastic gradient estimate at the stale point $z_k=x_{k-\delta_k}$.
We introduce the deviation
\begin{align*}
\mu_k:=v_{k+1}-\nabla f(x_k),     
\end{align*}
and the estimate error
\begin{align*}
\zeta_k := \nabla F(z_k;\xi_{k-\delta_k}) - \nabla f(z_k)    
\end{align*}
which is unbiased due to Assumption \ref{ass:pBCM}.

Based on above notations, we decompose $\mu_k$ as follows
\begin{align}\label{eq:mu_recursion_outline}
\begin{split}
\mu_k
=\beta_k\mu_{k-1}
+\underbrace{\beta_k(\nabla f(x_{k-1})-\nabla f(x_k))}_{\text{iterate drift}}+\underbrace{(1-\beta)(\nabla f(z_k)-\nabla f(x_k))}_{\text{staleness bias}}
+\underbrace{(1-\beta)\zeta_k}_{\text{gradient noise}}.
\end{split}
\end{align}
According to Assumption \ref{ass:smooth} and Lemma \ref{lem:staleness},
we bound the iterate drift and the staleness bias as 
\begin{equation}\label{eq:iterate-drift}
\beta_t\|\nabla f(x_{t-1})-\nabla f(x_t)\|
\le \beta L\eta
\end{equation}
and  
\begin{equation}\label{eq:staleness-bias}
(1-\beta)\|\nabla f(z_t)-\nabla f(x_t)\|
\le (1-\beta) LR\eta.
\end{equation}
By unrolling \eqref{eq:mu_recursion_outline}, the gradient noise leads to the term $\sum_{t=1}^k (1-\beta)\beta^{k-t}\zeta_t$.
According to the $p$-BCM condition in Assumption \ref{ass:pBCM}, it holds
\begin{equation}\label{eq:noise-term}
\E \left[\left\|\sum_{t=1}^k (1-\beta)\beta^{k-t}\zeta_t\right\|\right] 
\le 2^{\frac{1}{p}}(1-\beta)^{\frac{p-1}{p}}\sigma.
\end{equation}

Combining equations (\ref{eq:mu_recursion_outline})--(\ref{eq:noise-term}), 
we achieve the following upper bound on $\BE[\|\mu_k\|]=\E\big[\|v_{k+1}-\nabla f(x_k)\|\big]$.
\begin{lem}
\label{lem:mom_dev_tight}
Under Assumptions~\ref{ass:smooth}--\ref{ass:pBCM}, 
Algorithm \ref{alg:ringmaster_nsgd_mom_server}
holds 
\begin{align*}
\E\big\|v_{k+1}-\nabla f(x_k)\big\|
\le
\frac{L\eta}{1-\beta} + LR\eta + 2^{\frac{1}{p}}\sigma(1-\beta)^{\frac{p-1}{p}}
\end{align*}
for all $k\geq 1$.
\end{lem}

Combining Lemma~\ref{lem:descent_mom} with Lemma~\ref{lem:mom_dev_tight}
and telescoping over $k=0,\ldots,K-1$ yields  
\begin{align*}
\begin{split}
\frac{1}{K}\!\sum_{k=0}^{K-1}\!\E\|\nabla f(x_k)\|
\le\frac{\Delta}{\eta K}
\!+\!2L\eta\left(\frac{1}{\alpha}\!+\!R\!+\!1\right)
\!+\!2^{\frac{p+1}{p}}\sigma\alpha^{\frac{p-1}{p}}\!.
\end{split}
\end{align*}
Following the parameter settings in Theorem~\ref{thm:main_mom_tight},
the above inequality implies
\begin{align*}
\E\|\nabla f(\hat x)\|
=\frac{1}{K}\sum_{k=0}^{K-1}\E\|\nabla f(x_k)\|\le \epsilon,
\end{align*}
which achieves the desired result.

\paragraph{Discussion}
The main technical novelty in our analysis for proposed RANSGD-m (Algorithm \ref{alg:ringmaster_nsgd_mom_server}) is the decomposition of $\mu_k$ shown in equation  (\ref{eq:mu_recursion_outline}), which differs from that used in the proofs of previous methods in the following respects.
\begin{itemize}[leftmargin=0.5cm,topsep=-0.15cm,itemsep=-0.1cm]
    \item Recall that the existing analysis for the synchronized normalized SGD with momentum on single machine decomposes $\mu_k$ as follows \citep{hubler2025}
    \begin{align}\label{eq:muk-single}
    \begin{split}
        \mu_k =\beta_k\mu_{k-1} + \beta_k(\nabla f(x_k)-\nabla f(x_{k-1}))+ (1-\beta_k)(\nabla F(x_k;\xi_{k}) - \nabla f(x_k)).
    \end{split}
     \end{align}
     In contrast, the analysis of our RANSGD-m requires to consider the staleness caused by the asynchronous update, so that we additionally introduces the following staleness bias term
     \begin{align*}
     (1-\beta)(\nabla f(z_k)-\nabla f(x_k))    
     \end{align*}
     in decomposition (\ref{eq:mu_recursion_outline}).
     Accordingly, the gradient noise term in our analysis is
     \begin{align*}
         \zeta_k = \nabla F(z_k;\xi_{k-\delta_k}) - \nabla f(z_k), 
     \end{align*}
     which is with respect to the stale point $z_k=x_{k-\delta_k}$, 
     rather than the current iterate $x_k$ used in the last term of decomposition (\ref{eq:muk-single}).
     \item The existing analysis  the delay threshold protocol in asynchronous  optimization considers the iteration
     \begin{align*}
        x_{k+1} \gets x_k - \eta \nabla F(z_{k};\xi_{k-\delta_k}),    
     \end{align*}
     which is a variant of standard SGD. 
     As a comparison, our iteration scheme (\ref{eq:main-steps}) includes the momentum and the normalization, which leads to the analysis be quite different and more complicated.
\end{itemize}

\section{The Lower Complexity Bounds}\label{sec:lower-bound}

This section provides lower bounds to show the proposed RANSGD-m (Algorithm \ref{alg:ringmaster_nsgd_mom_server}) is optimal for both fixed and universal computation models with heavy-tailed noise.

\subsection{The Fixed Computation Model}

We first focus on the fixed computation model corresponding to Assumption \ref{ass:time}, i.e., each worker $i$ has a fixed upper bound $\tau_i>0$ on the time complexity to access a stochastic gradient estimate. 
We consider the algorithm class by following the definition of \citet{tyurin2023optimal} 
but generalize the order of central moment for gradient noise from $p=2$ to $p\in(1,2]$.

\begin{dfn}
\label{dfn:time_multi_oracle}
An asynchronous stochastic first-order oracle algorithm $\fA=\{A_k\}_{k\ge 0}$ over $n$ workers under the fixed computation model satisfies the following constraints.
\begin{itemize}[leftmargin=0.3cm,topsep=-0.1cm,itemsep=-0.1cm]
\item Each worker $i\in[n]$ maintains an internal state
\[
s_i=(s_{t,i}, s_{x,i}, s_{q,i})\in \R_{\ge 0}\times \R^{d} \times \{0,1\},
\]
where $s_{t,i}\in\BR$ stores  start time of 
current computation,
$s_{x,i}\in\BR^{d}$ stores the corresponding requested point,
and $s_{q,i}$ stores the state of worker $i$, i.e., $s_{q,i}=0$ when worker $i$ is idle and $s_{q,i}=1$ when worker $i$ is busy.
\item Each worker $i\in[n]$ is associated with the time sensitive stochastic first-order oracle $O^{\widehat{\nabla} F}_{\tau_i}$, which takes input as the current time $t$, requested point $x$, and current state~$s_i$ and output  $(s_{i,+},g)=O^{\widehat{\nabla} F}_{\tau_i}(t,x,s_i)$ such that 
\begin{align*}
\begin{split}    
&~\bigl(s_{i,+}, g\bigr)=
O^{\widehat{\nabla} F}_{\tau_i}(t,x,s_i):= 
\begin{cases}
\bigl((t,x,1),0\bigr),
&s_{q,i}=0,\\[2pt]
\bigl((s_{t,i},s_{x,i},1),0\bigr),
&s_{q,i}=1 \ \text{and}\ t<s_{t,i}+\tau_i,\\[2pt]
\bigl((0,0,0),\widehat{\nabla} F(s_{x,i};\xi)\bigr),
&s_{q,i}=1 \ \text{and}\ t\ge s_{t,i}+\tau_i,
\end{cases}
\end{split}
\end{align*}
where $\xi$ is sampled from distribution~$\fD$ and $\widehat{\nabla} F(s_{x,i};\xi)$ 
can be accessed by the fixed computation model described by Assumption \ref{ass:time} and 
$\mathbb{E}[\widehat{\nabla}F(s_{x,i};\xi)]=\nabla f(s_{x,i})$ and 
$\mathbb{E}[\|\widehat{\nabla} F(s_{x,i}; \xi) - \nabla f(s_{x,i}) \|^p] \le \sigma^p$ for some $p\in(1,2]$ and $\sigma>0$. 
\item The asynchronous stochastic first-order
oracle algorithm $\fA=\{A_k\}_{k\ge 0}$ initializes $t_0=0$ and $x_0\in\R^{d}$, then performs oracles $\{O^{\widehat{\nabla} F}_{\tau_i}\}_{i=1}^n$ 
during the iterations. 
At round $k$, after observing gradient estimates $\{g_1,\dots,g_k\}$ from the outputs of the past oracle calls, 
the algorithm perform~$A_k$ to select a wall-clock time $t_{k+1}$ such that $t_{k+1}\ge t_k$ and worker $i_{k+1}$ to output
\begin{align*}    
(t_{k+1}, i_{k+1}, x_k)\ :=\ A_k(g_1,\dots,g_k),
\end{align*}
where $x_k\in\BR^d$ satisfies
$\mathrm{supp}(x_k) \ \subseteq \ \bigcup_{s \in [k]} \mathrm{supp}(g_s)$.
Here, we use the notation $\mathrm{supp}(\cdot)$ the present supporting set, i.e., for given vector $v=[v_{(1)},\dots,v_{(d)}]^\top\in\BR^d$, we denote $\mathrm{supp}(v):=\{j:\ v_{(j)}\neq 0\}$. 
\end{itemize}
\end{dfn}

For the algorithm class shown in Definition \ref{dfn:time_multi_oracle}, we establish the following lower bound on the time complexity for asynchronous stochastic nonconvex optimization under the fixed computation model with heavy-tailed noise.

\begin{thm}\label{thm:lb_asgd_pBCM}
For any $L>0$, $\Delta>0$, $p\in(1,2]$, $\sigma>0$, $\{\tau_i>0\}_{i=1}^n$, and $\epsilon\le c\sqrt{L\Delta}$ with some constant $c>0$, 
there exists an $L$-smooth function $f:\BR^d\to\BR$ associated with stochastic first-order oracles $\{O^{\widehat{\nabla} F}_{\tau_i}\}_{i=1}^n$ such that any asynchronous stochastic first-order oracle algorithm $\fA=\{A_k\}_{k\ge 0}$ follows Definition \ref{dfn:time_multi_oracle} with the initial point $x_0\in\BR^d$ such that $f(x_0)-\inf_{x\in\BR^d}f(x)\leq\Delta$ requires at least the time complexity of 
\begin{equation*}
\Omega\left(\min_{m\in[n]}
\left(\frac{1}{m}\sum_{i=1}^m \frac{1}{\tau_i}\right)^{-1}
\left(
\frac{L\Delta}{\epsilon^{2}}
+
\frac{L \sigma^{\frac{p}{p-1}} \Delta}{m\epsilon^{\frac{3p-2}{p-1}}}
\right)\right).
\end{equation*}
for finding an $\epsilon$-stationary point $\hat x\in\BR^d$ such that
\begin{align*}
\E\left[\Norm{\nabla f(\hat x)}\right] \le \epsilon.    
\end{align*}
\end{thm}

The lower bound shown in our Theorem \ref{thm:lb_asgd_pBCM} matches the upper bound provided in Corollary \ref{cor:time_mom_refined}, which indicates the proposed RANSGD-m (Algorithm \ref{alg:ringmaster_nsgd_mom_server}) is optimal to the computation model associated with the oracle satisfying the  $p$-BCM assumption.
In the case of $p=2$, our lower bound reduces to the tight one under the bounded variance setting \citep{tyurin2023optimal}.

\subsection{The Universal Computation Model}

For the universal computation model (Assumption \ref{ass:universalmodel}), the time complexity for achieving a stochastic gradient estimate on worker $i\in[n]$ is not fixed but depends on the computation power $\upsilon_i:\BR_+\to\BR_+$.

Hence, we modify the algorithm class in Definition~\ref{dfn:time_multi_oracle} by replacing the oracles $O^{\widehat{\nabla} F}_{\tau_i}(t,x,s_i)$ with
\begin{align*}
\begin{split}
&~~O^{\widehat{\nabla} F}_{V_i}(t,x,s_i):=\!\begin{cases}
\bigl((t,x,1),0\bigr),
&\!\!s_{q,i}=0,\\[2pt]
\bigl((s_{t,i},s_{x,i},1),0\bigr),
&\!\!s_{q,i}\!=\!1~\text{and}~V_i(t)\!-\!V_i(s_{t,i})\!<\!1,\\[2pt]
\bigl((0,0,0),\widehat{\nabla} F(s_{x,i};\xi)\bigr),
&\!\!s_{q,i}\!=\!1~\text{and}~V_i(t)\!-\!V_i(s_{t,i})\!\ge\!1,
\end{cases}
\end{split}
\end{align*}
where $\widehat{\nabla} F(s_{x,i};\xi)$ satisfies 
$\mathbb{E} [ \widehat{\nabla} F(s_{x,i}; \xi)] = \nabla f(s_{x,i})$ and 
$\mathbb{E}\left[\|\widehat{\nabla} F(s_{x,i}; \xi) - \nabla f(s_{x,i}) \|^p\right] \le \sigma^p$ 
for all some $\sigma>0$ and $p\in(1,2]$ and  
$V_i(t):=\int_{0}^{t} \upsilon_i(\tau)\,{\rm d}\tau$ is the cumulative computation work function of worker $i$.
We present details for the formal definition of such algorithm class considered in the universal computation model in Definition \ref{dfn:time_multi_oracle_universal} of Appendix~\ref{sec:lb_asgd_heavytail_universal}.

For the above algorithm class associated with the modified oracle $O^{\widehat{\nabla} F}_{V_i}(t,x,s_i)$,
we establish the following lower bound on the time complexity for asynchronous stochastic nonconvex optimization under the universal computation model with heavy-tailed noise.

\newpage 

\begin{table*}[t]
\centering
\caption{We compare the time complexity of proposed RANSGDm with existing asynchronous stochastic first-order algorithms for finding an $\epsilon$-stationary point.
Note that only our results provide the theoretical guarantees for the general $p\in(1,2]$.} \vskip0.2cm
\label{table:compare}
\begin{tabular}{cccc}
\hline
Methods & Time Complexity for Fixed Computation Model & Noise &  Adaptive to UCM \\ \hline\addlinespace
\makecell{ASGD \\[0.1cm] \citep{mishchenko2022asynchronous}} &  $\displaystyle{\fO\left(\left(\frac{1}{n}\sum_{i=1}^n \frac{1}{\tau_i}\right)^{-1} \left(\frac{L\Delta}{\epsilon^{2}} + \frac{L\sigma^{2}\Delta}{n\epsilon^{4}}\right)\right)}$ & $p=2$ & \XSolidBrush \\\addlinespace  
\makecell{Rennala ASGD \\[0.1cm] \citep{tyurin2023optimal}} &  $\displaystyle{\fO\left(\min_{m\in[n]}\left(\frac{1}{m}\sum_{i=1}^m \frac{1}{\tau_i}\right)^{-1} \left(\frac{L\Delta}{\epsilon^{2}} + \frac{L\sigma^{2}\Delta}{m\epsilon^{4}}\right)\right)}$ & $p=2$ & \Checkmark \\ \addlinespace 
\makecell{Ringmaster ASGD \\[0.1cm] \citep{pmlr-v267-maranjyan25b}} &  $\displaystyle{\fO\left(\min_{m\in[n]}\left(\frac{1}{m}\sum_{i=1}^m \frac{1}{\tau_i}\right)^{-1} \left(\frac{L\Delta}{\epsilon^{2}} + \frac{L\sigma^{2}\Delta}{m\epsilon^{4}}\right)\right)}$ & $p=2$ & \Checkmark \\ \addlinespace 
\makecell{RANSGDm \\[0.1cm] (Corollary \ref{cor:time_mom_refined})} &  $\displaystyle{\fO\left(\min_{m\in[n]}\left(\frac{1}{m}\sum_{i=1}^m \frac{1}{\tau_i}\right)^{-1} \left(\frac{L\Delta}{\epsilon^{2}} + \frac{L\sigma^{\frac{p}{p-1}}\Delta}{m\epsilon^{\frac{3p-2}{p-1}}}\right)\right)}$ & $p\in(1,2]$ & \Checkmark \\ \addlinespace \hline\addlinespace
\makecell{Lower Bound \\[0.1cm] \citep{tyurin2023optimal}} &  $\displaystyle{\Omega\left(\min_{m\in[n]}\left(\frac{1}{m}\sum_{i=1}^m \frac{1}{\tau_i}\right)^{-1} \left(\frac{L\Delta}{\epsilon^{2}} + \frac{L\sigma^{2}\Delta}{m\epsilon^{4}}\right)\right)}$ & $p=2$ & -- \\ \addlinespace \addlinespace
\makecell{Lower Bound \\[0.1cm] (Theorem \ref{thm:lb_asgd_pBCM})} &  $\displaystyle{\Omega\left(\min_{m\in[n]}\left(\frac{1}{m}\sum_{i=1}^m \frac{1}{\tau_i}\right)^{-1} \left(\frac{L\Delta}{\epsilon^{2}} + \frac{L\sigma^{\frac{p}{p-1}}\Delta}{m\epsilon^{\frac{3p-2}{p-1}}}\right)\right)}$ & $p\in(1,2]$ & -- \\ \addlinespace \hline
\end{tabular}
\end{table*}

\begin{thm}
\label{thm:lb_asgd_pBCM_universal}
For any $L>0$, $\Delta>0$, $p\in(1,2]$, $\sigma>0$, $\{v_i>0\}_{i=1}^n$, and $\epsilon\le c\sqrt{L\Delta}$ with some constant $c>0$, 
there exists an $L$-smooth function $f:\BR^d\to\BR$ associated with stochastic stochastic first-order oracles $\{O^{\widehat{\nabla} F}_{V_i}\}_{i=1}^n$ such that any asynchronous stochastic first-order oracle algorithm $\fA=\{A_k\}_{k\ge 0}$ follows the universal computation model (see details in Definition \ref{dfn:time_multi_oracle_universal}) with the initial point $x_0\in\BR^d$ such that $f(x_0)-\inf_{x\in\BR^d}f(x)\leq\Delta$ requires at least the time complexity of $T_{\tilde K}$ for finding an $\epsilon$-stationary point $\hat x\in\BR^d$ such that
\begin{align*}
\E\left[\Norm{\nabla f(\hat x)}\right] \le \epsilon.    
\end{align*}
Here we define $\tilde K$ as
\begin{align*}
    \tilde K &:= \left\lfloor \frac{1}{2}\left\lfloor\frac{L\Delta}{29184\epsilon^{2}}\right\rfloor+\log\left(\frac{1}{2}\right)\right\rfloor = \Omega\left( \frac{L\Delta}{\epsilon^2} \right),
\end{align*}
and $T_{\tilde K}$ is the $\tilde K$th element of the following recursively defined sequence $\{T_K\}_{K\ge 0}$ such that $T_0= 0$ and
\begin{align*}
    T_K &:= \min\left\{ T \ge 0 : \sum_{i=1}^{n} \int_{T_{K-1}}^{T} \upsilon_i(\tau) \,{\rm d}\tau \ge \left\lceil\frac{1}{4q}\right\rceil \right\}
\end{align*}
for all $K\geq 1$, where
$q=\min\big\{\big({92\cdot 2^{\frac1p}\epsilon}/\sigma\big)^{\frac{p}{p-1}},1\big\}$ which leads to $\left\lceil1/(4q)\right\rceil=\Theta\big(\lceil\sigma/\epsilon\rceil^{\frac{p}{p-1}}\big)$.
\end{thm}

The order of $\tilde K$ and $T_K$ the Theorem \ref{thm:lb_asgd_pBCM_universal} align with the counterparts of $\bar K$ and $T_K$ in Corollary \ref{lem:universal_time_recursion_blocks}.
Therefore, the proposed RANSGD-m (Algorithm \ref{alg:ringmaster_nsgd_mom_server}) is also optimal to the universal computation model.
In the case of $p=2$, the lower bound match the results of the bounded variance setting \citep{tyurin2023optimal}.
We compare our results with related work in Table \ref{table:compare}.

\begin{figure*}[ht]
    \centering
    \begin{tabular}{ccc}
      \includegraphics[scale=0.3]{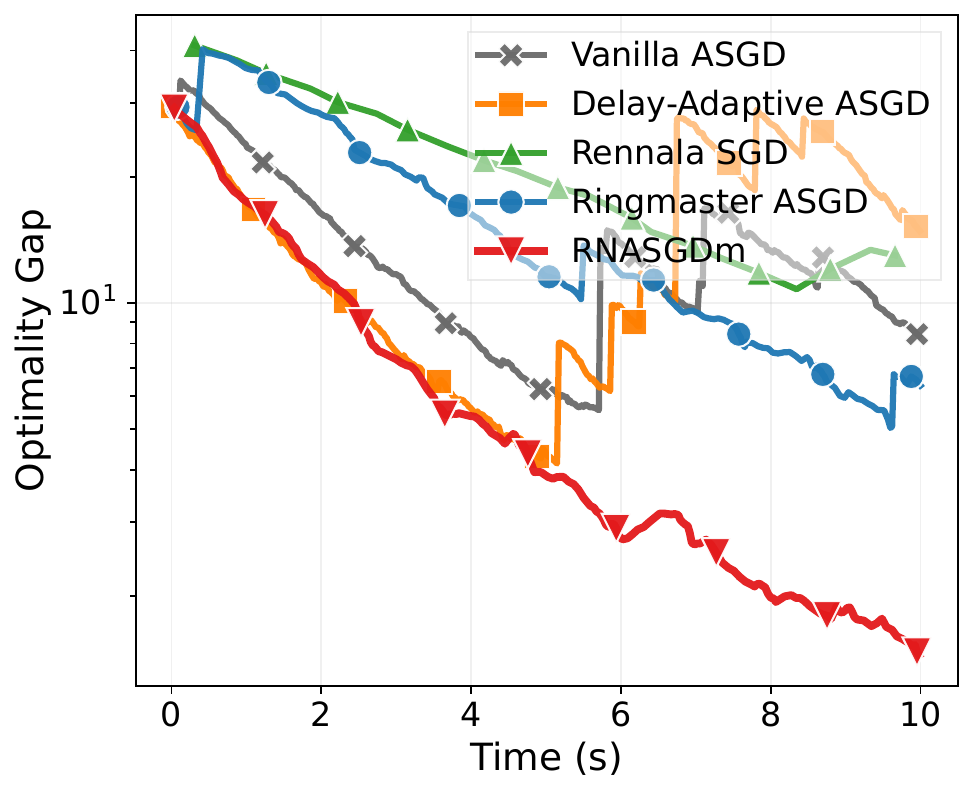}  &  
      \includegraphics[scale=0.3]{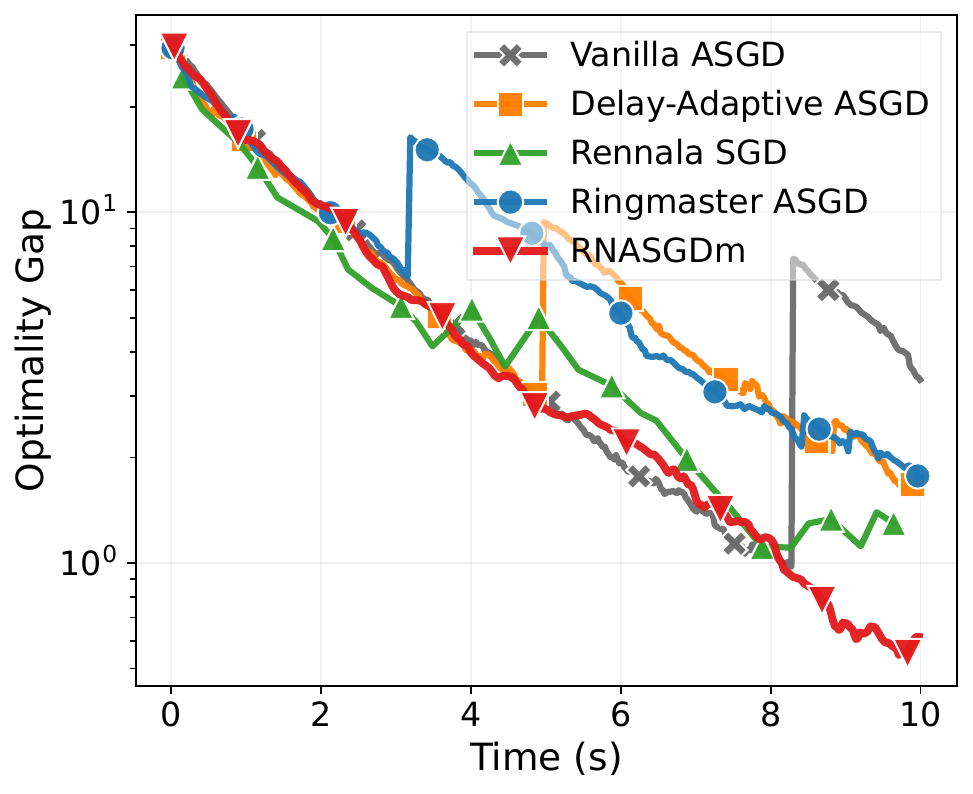}  &  
      \includegraphics[scale=0.3]{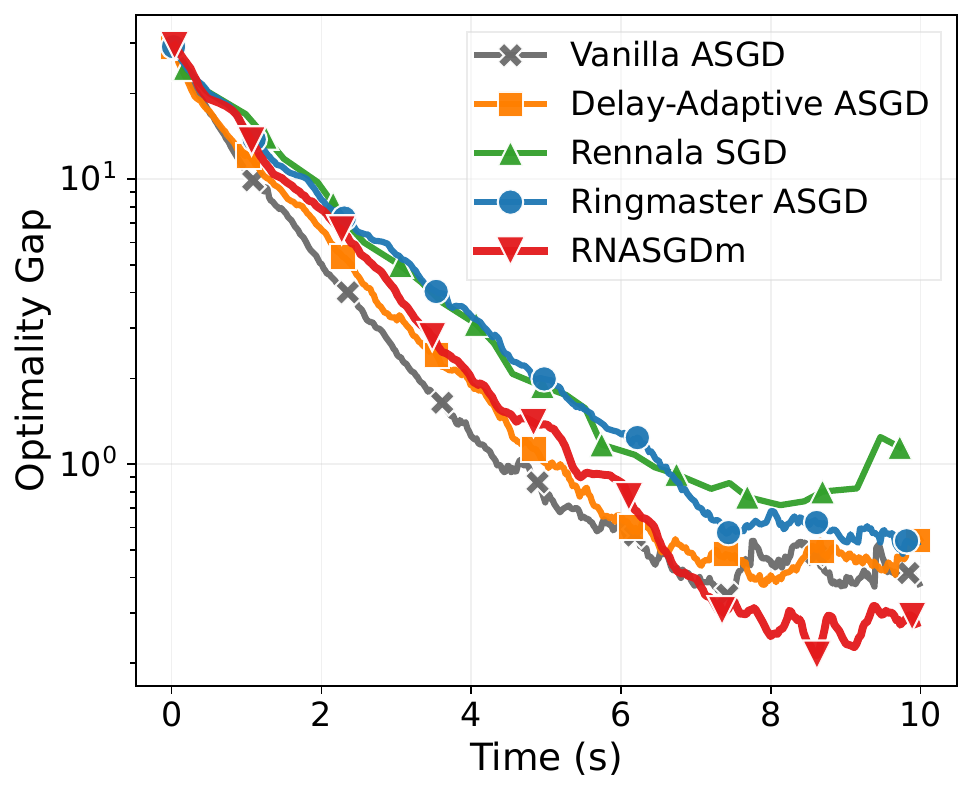}  \\        
      (a) $\nu=1.5$ & (b) $\nu=1.8$ & (c) $\nu=2.1$
    \end{tabular}  \vskip-0.2cm
    \caption{The results of solving the quadratic problem, where the  expected delay of slow workers is 0.02s.}
    \label{fig:quad_slow002}
\end{figure*} 

\begin{figure*}[ht]
    \centering
    \begin{tabular}{ccc}
      \includegraphics[scale=0.3]{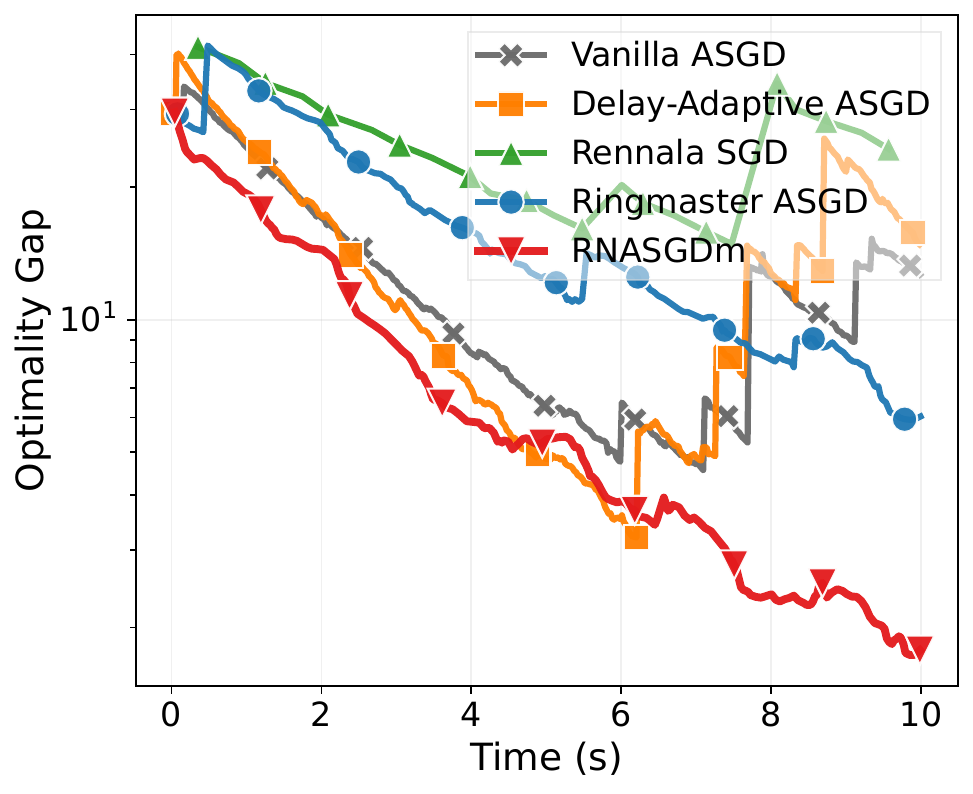}  &  
      \includegraphics[scale=0.3]{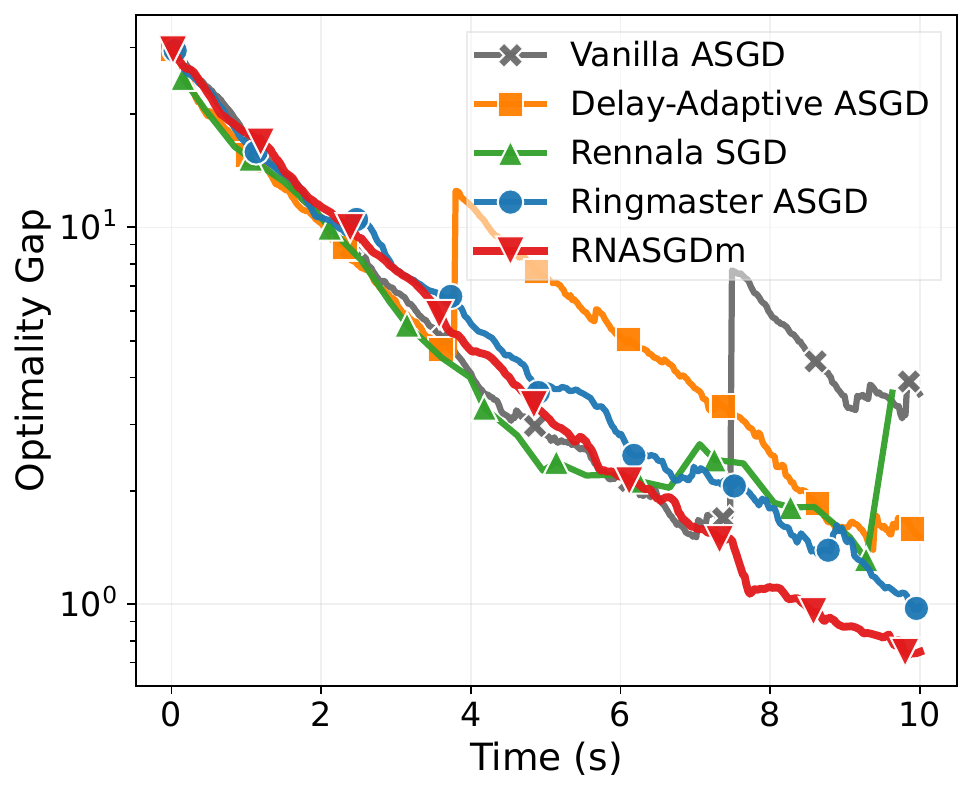}  &  
      \includegraphics[scale=0.3]{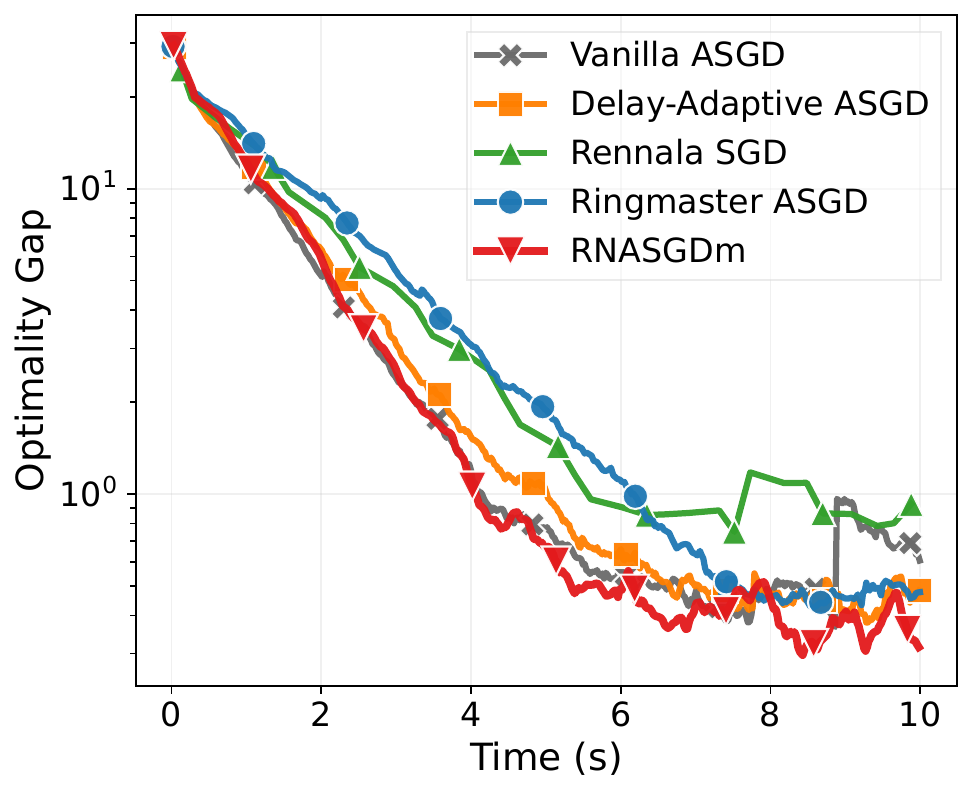}  \\      
      (a) $\nu=1.5$ & (b) $\nu=1.8$ & (c) $\nu=2.1$
    \end{tabular}  \vskip-0.2cm
    \caption{The results of solving the quadratic problem, where the  expected delay of slow workers is 0.005s.}
    \label{fig:quad_slow0005} 
\end{figure*}

\begin{figure*}[ht]
    \centering
    \begin{tabular}{ccc}
      \includegraphics[scale=0.3]{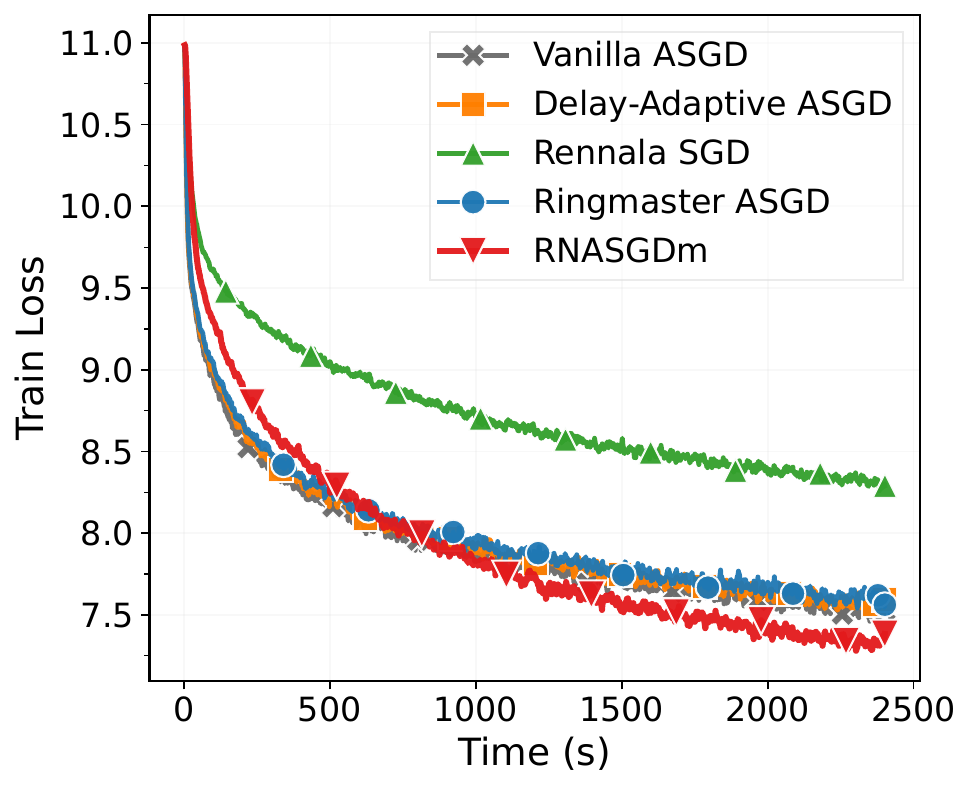}  &  
      \includegraphics[scale=0.3]{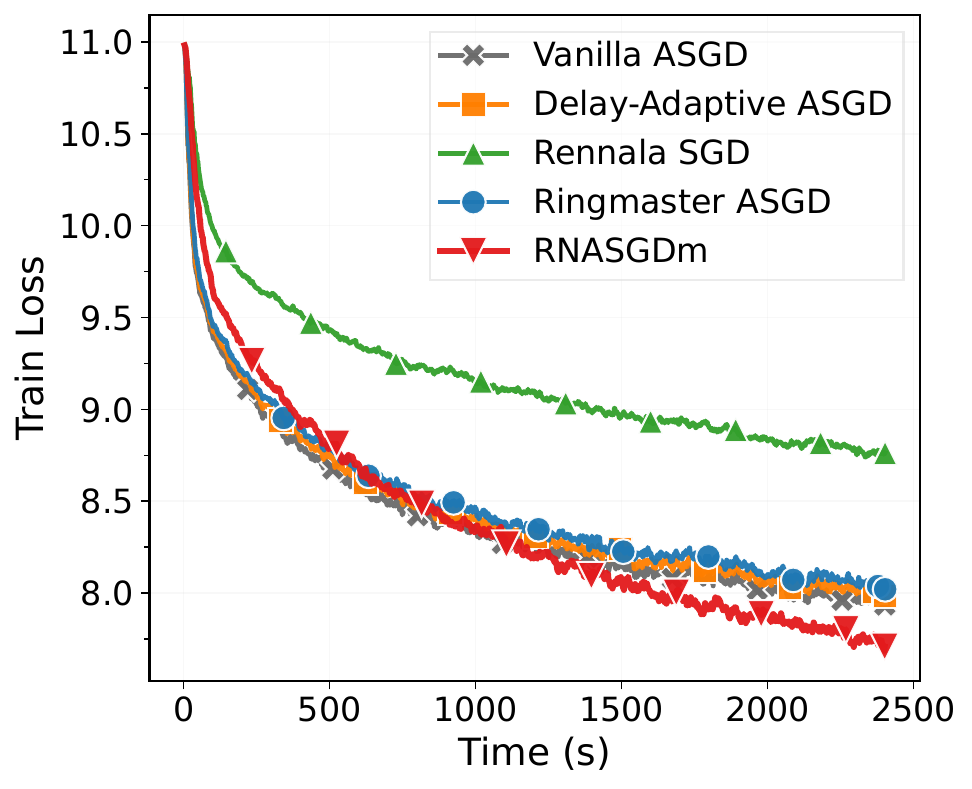}  &  
      \includegraphics[scale=0.3]{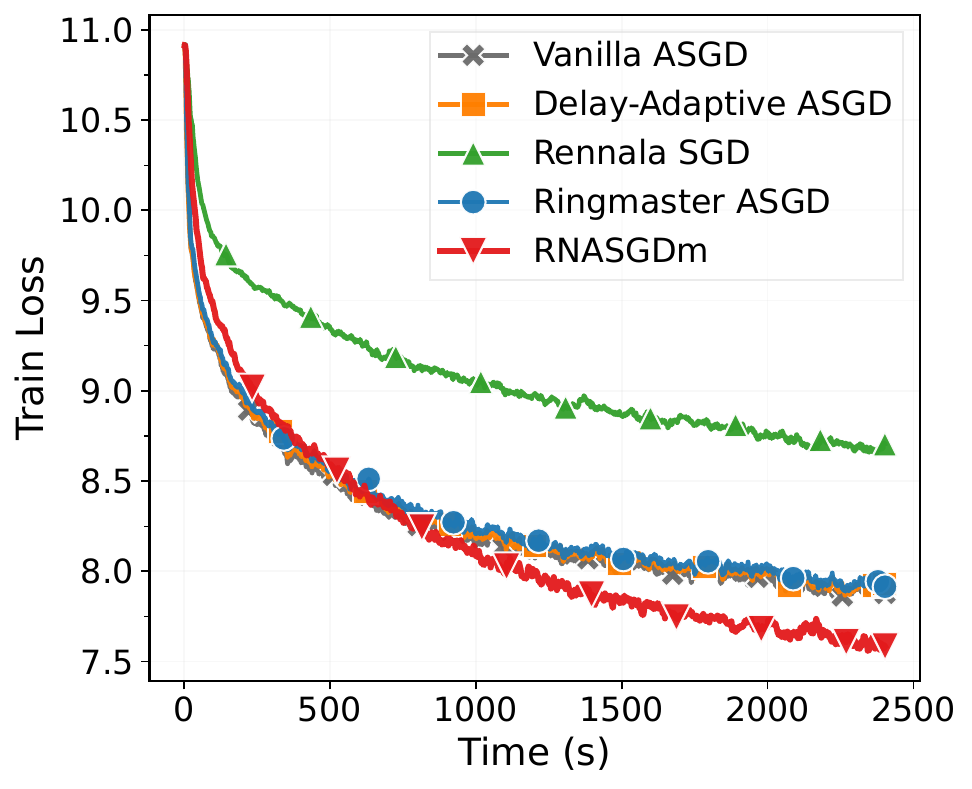}  \\    
      \footnotesize (a) exponential distributed delays (I) & 
      \footnotesize (b) exponential distributed delays (II) & 
      \footnotesize (c) Pareto distributed delays
    \end{tabular}
    \caption{The results of training GPT-2 model with different delay settings. The sub-figures (a) and (b) set exponential distributed delays with expectations $\{0,0.5,0.9,1.2,1.5,2.0,2.5\}$ and $\{1,2,3,4,5,6,7\}$, respectively. The sub-figure (c) set Pareto distributed delays with expectations $\{1,2,3,4,5,6,7\}$.}
    \label{figure:GPT-1} \vskip-0.4cm
\end{figure*}

\section{Experiments}
\label{sec:experiments}

We compare our proposed RANSGDm (Algorithm \ref{alg:ringmaster_nsgd_mom_server}) with baseline methods including Vanilla ASGD, Delay-Adaptive ASGD \citep{mishchenko2022asynchronous}, Rennala SGD \citep{tyurin2023optimal}, and Ringmaster ASGD \citep{pmlr-v267-maranjyan25b} on solving the synthetic quadratic problem 
and the task of training GPT-2 model.
We implement all algorithms by the framework of Ray \citep{moritz2018ray}.

\subsection{The Quadratic Problem}
\label{subsec:exp_quadratic}

We consider the optimization problem of minimizing the quadratic function 
$f(x)=\frac{1}{2}x^\top A x-b^\top x$    
on the distributed system consisting of one central server and $n=40$ workers,
where $A=\frac{1}{20000}X^\top X+0.01I$ with  $X\in\mathbb{R}^{20000\times 50}$ such that its entries are independently sampled from the standard Gaussian distribution and we set $b=Ax^\star\in\BR^{50}$ by taking~$x^\star\sim\mathcal{N}_{50}(0,I)$.

We set the stochastic gradient as $\nabla F(x;\xi)=\nabla f(x)+\zeta$ for all algorithm in experiments, where the noise vector~$\zeta$ has independent coordinates drawn from a Student-$t$ distribution with the degrees of freedom $\nu$.
Note that smaller $\nu$ leads to the more heavy-tailed noise.
We let $\nu\in\{1.5,1.8,2.1\}$ for our experiments.
We simulate the heterogeneous computation delay by following the setting of \citet{mishchenko2022asynchronous}, i.e., 
all 40 workers are split into 20 ``fast workers'' and 20 ``slow workers''.
For ``faster workers'', we set the delay to follow an exponential distribution with expectation of $0.001$s.
For ``slower workers'', we set the delay to follow an exponential distribution with expectation of $0.02$s and $0.005$s in two different settings.
For all algorithms, we tune the stepsize from $\{10^{-4},2\times10^{-4},\cdots,9\times10^{-2}\}$.
For Rennala SGD, we set the collection batch size to 6.
For Ringmaster ASGD and RANSGDm, we tune $R$ from $\{6,8,\dots,20\}$.
We fix momentum parameter $\beta=0.9$ for our RANSGDm.

We report results of optimality gap versus wall-clock time in 
Figures \ref{fig:quad_slow002} and \ref{fig:quad_slow0005} with two different delay settings for slower workers.
We observe our RANSGDm consistently achieves the best convergence behavior.
Moreover, advantage of RANSGDm is more clear under the heavy-tailed noise setting, i.e., $\nu=1.5$ and $\nu=1.8$, while the performance of all algorithms is close when the noise is not heavy-tailed, i.e., $\nu=2.1$.

\subsection{Training GPT-2}
\label{subsec:exp_gpt2}
We conduct the experiments on training GPT-2 model (124M parameters) on OpenWebText dataset \cite{Gokaslan2019OpenWeb}. 
We implement all algorithms on the systems of 8 RTX4090 GPUs, i.e., one server and $n=7$ workers. 
We use mini-batch size $12$ and context length $1024$ for all algorithms. 
We consider three different delay settings on workers, i.e., the exponential distributions with expectations $\{0,0.5,0.9,1.2,1.5,2.0,2.5\}$ seconds,
the exponential distributions with expectations $\{1,2,3,4,5,6,7\}$ seconds, 
and the Pareto distributions with expectations $\{1,2,3,4,5,6,7\}$ seconds.
For all algorithms, we tune stepsizes from $\{10^{-5}, 2\times10^{-5}$, 
$\cdots,5\times10^{-4}\}$.
For Rennala SGD, we set the collection batch size to 4.
For Ringmaster ASGD and RANSGDm, we set $R=8$.
We fix momentum parameter $\beta=0.9$ for our RANSGDm.

We report the results of training loss versus wall-clock time in Figure \ref{figure:GPT-1}.
We observe our RANSGDm consistently has better performance than baseline methods, which supports our theoretical analysis that our methods can well address the widely observed heavy-tailed noise in the language model training.

\section{Conclusion}

In this paper, we study asynchronous stochastic nonconvex optimization with arbitrarily heterogeneous computation times across workers.
We focus on the setting of heavy-tailed gradient noise under the $p$-BCM assumption and propose an asynchronous normalized
stochastic gradient descent with momentum and delay threshold.
Our analysis shows that the proposed method has the optimal time complexity for both the fixed computation model and the universal computational model. We also provide empirical studies on synthetic problem and training the language model to validate the advantage of our method.

In future work, we are interested in extending our ideas to design  asynchronous stochastic algorithms for solving the nonconvex nonsmooth optimization problem. 
We also would like to establish asynchronous stochastic algorithms for decentralized optimization with heavy-tailed noise. 

\bibliographystyle{plainnat}
\bibliography{reference}

\newpage
\appendix

\section{The Proofs for Results in Section \ref{sec:alg-main}}
\label{sec:ringmaster_nsgd_momentum}

We first provide two supporting lemmas for our later analysis.
\begin{lem}[{\citet[Lemma 7]{hubler2025}}]
\label{lem:angle}
For all $a,b \in \mathbb{R}^d$ with $b \neq 0$, we have
$$
\frac{a^\top b}{\| b \|} \ge \|a\| - 2\|a-b\|.
$$
\end{lem}

\begin{lem}[{\citet[Lemma~10]{hubler2025}}]
\label{lem:vbe_mds}
Let $p\in[1,2]$, and $X_1,\dots,X_n\in\R^d$ be a martingale difference sequence (MDS), i.e.,
\begin{align*}
\E\!\left[X_j\,\middle|\,X_{j-1},\dots,X_1\right]=0
\quad \text{a.s. for all } j=1,\dots,n
\end{align*}
satisfying
\begin{align*}
\E\!\left[\Norm{X_j}^p\right]<\infty
\qquad \text{for all } j=1,\dots,n.
\end{align*}
Define $S_n:=\sum_{j=1}^n X_j$, then
\begin{align*}
\E\!\left[\Norm{S_n}^p\right]\le 2\sum_{j=1}^n \E\!\left[\Norm{X_j}^p\right].
\end{align*}
\end{lem}

Recall that the update on the server can be written as

\begin{empheq}[left=\empheqlbrace]{align}
v_{k+1} &= \beta_k v_k + (1-\beta) \nabla F(x_{k-\delta_k};\xi_{k-\delta_k}), \nonumber \\[0.2cm]
x_{k+1} &= x_k - \dfrac{\eta v_{k+1}}{\|v_{k+1}\|}, \label{eq:x_update}
\end{empheq}
where 
\begin{align*}
    \beta_k \gets \begin{cases}
          0,  & k\leq 1 \\[0.02cm]
          \beta,  & k > 1,
    \end{cases}
    \qquad \beta \in [0,1),
    \qquad \eta>0,
    \qquad\text{and}\qquad
    \delta_k<R.
\end{align*}
where $z_k:=x_{k-\delta_k}$ with $\delta_k\le R$.

We initialize $v_0:=0$ and define 
\begin{align*}
    z_k:=x_{k-\delta_k},
\end{align*}
which leads to $z_0=x_0$ and
\begin{equation}\label{eq:mom_update_clean}
v_{k+1} \;=\; \beta_k v_k + (1-\beta) \nabla F(z_k;\xi_{k-\delta_k}).
\end{equation}
We further define the noise at the stale point as
\begin{align}
\label{eq:zetak_clean}
\zeta_k:=\nabla F(z_k;\xi_{k-\delta_k})-\nabla f(z_k),
\end{align}
which is unbiased for given $z_k$.

We then provide the proofs in Section \ref{sec:alg-main} for our upper bounds.

\subsection{The Proof of Lemma \ref{lem:staleness}}

\begin{proof}
According to the update rule \eqref{eq:x_update}
we have
\begin{align*}
    \|x_k - x_{k-\delta_k}\| 
     =  \left\| \sum_{j=k-\delta_k}^{k-1} (x_{j+1} - x_j) \right\| 
     \le  \sum_{j=k-\delta_k}^{k-1} \eta \Norm{\frac{v_{j+1}}{\|v_{j+1}\|}} 
    = \eta \cdot \delta_k \leq R\eta,
\end{align*}
where first inequality is based on the triangle inequality and the last inequality is based on the algorithm setting that the update only performs when $\delta_k \leq R$.
\end{proof}

\subsection{The Proof of Lemma \ref{lem:descent_mom}}

\begin{proof}
Based on Assumption \ref{ass:smooth} and the update rule (\ref{eq:x_update}), we have
\begin{align*}
f(x_{k+1})
&\le f(x_k)+\inner{\nabla f(x_k)}{x_{k+1}-x_k}+\frac{L}{2}\|x_{k+1}-x_k\|^2\\
&= f(x_k)-\eta\inner{\nabla f(x_k)}{\frac{v_{k+1}}{\|v_{k+1}\|}}+\frac{L\eta^2}{2}\\
&\le f(x_k)
-\eta\|\nabla f(x_k)\|
+2\eta\|v_{k+1}-\nabla f(x_k)\|
+\frac{L\eta^2}{2}.
\end{align*}
where the last inequality is based on Lemma \ref{lem:angle} with $a=\nabla f(x_k)$ and $b=v_{k+1}$. 
\end{proof}

\subsection{The Proof of Lemma \ref{lem:mom_dev_tight}}

\begin{proof}
We define 
\begin{align*}
\mu_k:=v_{k+1}-\nabla f(x_k).
\end{align*}
For all $k\ge 1$, combining equations (\ref{eq:mom_update_clean}) and (\ref{eq:zetak_clean}) leads to
\begin{align}
\label{eq:mu_decomp_clean}
\mu_k
&= \beta_k v_k + (1-\beta)(\nabla f(z_k)+\zeta_k) - \nabla f(x_k) \nonumber\\
&= \beta_k\big(v_k-\nabla f(x_{k-1})\big)
+ \beta_k\big(\nabla f(x_{k-1})-\nabla f(x_k)\big)
+ (1-\beta)\big(\nabla f(z_k)-\nabla f(x_k)\big)
+ (1-\beta)\zeta_k \nonumber\\
&= \beta_k \mu_{k-1}
+ \underbrace{\beta_k\big(\nabla f(x_{k-1})-\nabla f(x_k)\big)}_{A_k}
+ \underbrace{(1-\beta)\big(\nabla f(z_k)-\nabla f(x_k)\big)}_{B_k}
+ \underbrace{(1-\beta)\zeta_k}_{C_k}.
\end{align}
Hence, we can write \eqref{eq:mu_decomp_clean} as
\begin{align}\label{eq:recursion-mu-k}
\mu_k = \beta_k \mu_{k-1} + A_k + B_k + C_k.
\end{align}
For $k=1$, we have $\beta_1=0$,
then \eqref{eq:mu_decomp_clean} implies
\begin{align}\label{eq:recursion-mu-0}
\mu_1 = \beta_1\mu_0 + A_1 + B_1 + C_1 = A_1 + B_1 + C_1.
\end{align}

Combining equations (\ref{eq:recursion-mu-k}) and (\ref{eq:recursion-mu-0}) yields
\begin{equation}
\label{eq:mu_sum_reduced_clean}
\mu_k
= \sum_{t=1}^k \beta^{k-t}(A_t+B_t+C_t)
\end{equation}
for all $k\geq 1$.

We then consider the norms of $A_t$, $B_t$, and $C_t$, respectively.
\begin{itemize}
    \item For the term $A_t$, we have
\begin{equation}
\label{eq:A_bound_chain}
\|A_t\|
= \beta_t\|\nabla f(x_{t-1})-\nabla f(x_t)\|
\le \beta_t L\|x_{t-1}-x_t\|
= \beta_t L\eta
\le \beta L\eta,
\end{equation}
the first inequality is based on Assumption \ref{ass:smooth};
the second equality is based on the update rule (\ref{eq:x_update});
and the last step is based on the fact $\beta_t\le \beta$ since $\beta_1=0$ and $\beta_t=\beta$ for $t\ge 2$.
\item For the term $B_t$, we have
\begin{equation}
\label{eq:B_bound_chain}
\|B_t\|
= (1-\beta)\|\nabla f(z_t)-\nabla f(x_t)\|
\le (1-\beta) L\|z_t-x_t\|
\le (1-\beta) LR\eta.
\end{equation}
where the first inequality is based on Assumption \ref{ass:smooth} 
and the second inequality is based on Lemma \ref{lem:staleness}.
\item For the term $C_t$, we have
\begin{align}\label{eq:noise_l1_chain0}
\begin{split}    
\E\left[\Norm{\sum_{t=1}^k \beta^{k-t}C_t}\right]
= \E\left[\Norm{\sum_{t=1}^k (1-\beta)\beta^{k-t}\zeta_t}\right]
\le  \left(\E\left\|\sum_{t=1}^k (1-\beta)\beta^{k-t}\zeta_t \right\|^p\right)^{\frac1p},
\end{split}
\end{align}
where we use Jensen's inequality.
Let $\{\mathcal{F}_t\}_{t\ge 0}$ be the filtration generated by the server history up to round $t$, defined as
\begin{align*}
\mathcal{F}_t:=\sigma\!\left(x_0,\ z_1,\dots,z_t,\ \zeta_1,\dots,\zeta_t\right).
\end{align*}
Then $z_t$ is $\mathcal{F}_{t-1}$-measurable and $\zeta_t$ is revealed at round $t$.
Hence, by Assumption~\ref{ass:pBCM},
\begin{align*}
\E[\zeta_t\mid \mathcal{F}_{t-1}] = \E[\zeta_t\mid z_t]=0.
\end{align*}
Therefore, $\{(1-\beta)\beta^{k-t}\zeta_t\}_{t\ge 1}$ is a martingale difference sequence with respect to $\{\mathcal{F}_t\}_{t\ge 0}$.
Hence, we apply Lemma~\ref{lem:vbe_mds} to obtain
\begin{align}\label{eq:noise_l1_chain1}
\begin{split}    
\E\left[\left\|\sum_{t=1}^k (1-\beta)\beta^{k-t}\zeta_t\right\|^p\right]
&\le 2\sum_{t=1}^k \E\|(1-\beta)\beta^{k-t}\zeta_t\|^p
= 2(1-\beta)^p\sum_{t=1}^k \beta^{p(k-t)}\E\|\zeta_t\|^p \\
&\le 2(1-\beta)^p \sum_{j=0}^{k-1}\beta^{pj}\sigma^p
\le \frac{2(1-\beta)^p\sigma^p}{1-\beta^p},
\end{split}
\end{align}
where the second inequality is based on Assumption~\ref{ass:pBCM}.

Combining equations (\ref{eq:noise_l1_chain0}) and (\ref{eq:noise_l1_chain1}), we have
\begin{equation}
\label{eq:noise_l1_chain}
\E \left[\left\|\sum_{t=1}^k \beta^{k-t}C_t\right\|\right] 
\le \left(\frac{2(1-\beta)^p\sigma^p}{1-\beta^p}\right)^{\frac1p}
\le \left(\frac{2(1-\beta)^p\sigma^p}{1-\beta}\right)^{\frac1p}
=2^{\frac{1}{p}}(1-\beta)^{\frac{p-1}{p}}\sigma,
\end{equation}
where the second inequality is based on the settings of $\beta\in[0,1)$ and $p>1$ that lead to $1-\beta^p\ge 1-\beta$.
\end{itemize}
Combining equations (\ref{eq:mu_sum_reduced_clean}), (\ref{eq:A_bound_chain}), (\ref{eq:B_bound_chain}), and (\ref{eq:noise_l1_chain}), we have
\begin{align*}
\E\left[\|\mu_k\|\right]
&\le \E\left[\Norm{\sum_{t=1}^k \beta^{k-t}A_t}\right] + \E\left[\Norm{\sum_{t=1}^k \beta^{k-t}B_t}\right]
+ \E\left[\Norm{\sum_{t=1}^k \beta^{k-t}C_t}\right] \\
&\le \sum_{t=1}^k \beta^{k-t}\cdot\beta L\eta +  \sum_{t=1}^k \beta^{k-t}\cdot(1-\beta) LR\eta + 2^{\frac1p}\sigma(1-\beta)^{\frac{p-1}{p}} \\
&\le L\eta\sum_{t=1}^{k}\beta^{t} + (1-\beta) LR\eta\sum_{t=0}^{k-1}\beta^t  + 2^{\frac1p}\sigma(1-\beta)^{\frac{p-1}{p}} \\
&\le \frac{L\eta}{1-\beta} + LR\eta + 2^{\frac1p}\sigma(1-\beta)^{\frac{p-1}{p}},
\end{align*}
which finishes the proof.
\end{proof}

\subsection{The Proof of Theorem \ref{thm:main_mom_tight}}
\begin{proof}
Combining results of Lemmas~\ref{lem:mom_dev_tight} and \ref{lem:descent_mom}, we have
\begin{align*}
\begin{split}    
\E[f(x_{k+1})]
&\le \E[f(x_k)]-\eta\E\|\nabla f(x_k)\|
+2\eta\Big(\frac{L\eta}{\alpha}+LR\eta+ 2^{\frac1p}\sigma\alpha^{\frac{p-1}{p}}\Big)
+\frac{L\eta^2}{2} \\
&= \E[f(x_k)]-\eta\E\|\nabla f(x_k)\|
+2L\eta^2\Big(\frac{1}{\alpha}+R+\frac{1}{4}\Big)
+2^{\frac{p+1}{p}}\eta \sigma\alpha^{\frac{p-1}{p}}.
\end{split}
\end{align*}
Rearranging above inequality and taking the average over $k=0,\dots,K-1$, we have
\begin{align}\label{eq:avg_grad_bound_refined}    
\begin{split}    
\frac{1}{K}\sum_{k=0}^{K-1}\E\|\nabla f(x_k)\|
& \le \frac{f(x_0) - f(x_K)}{\eta K}
+ L\eta \left(\frac{2}{\alpha} + 2R + \frac{1}{2}\right)
+ 2^{\frac{p+1}{p}} \sigma\alpha^{\frac{p-1}{p}} \\
& \le \frac{\Delta}{{\eta K}}
+ 2L\eta \left(\frac{1}{\alpha} + R + 1\right)
+ 2^{\frac{p+1}{p}}\sigma\alpha^{\frac{p-1}{p}},
\end{split}
\end{align}

where the last step is based on  Assumption \ref{ass:smooth}.

For the right-hand side of equation (\ref{eq:avg_grad_bound_refined}), the settings of $\eta$, $K$, and $\beta$ implies
\begin{align}\label{eq:avg_grad_bound_1}
     \frac{\Delta}{\eta K} \leq \frac{\epsilon}{3}
     \qquad\text{and}\qquad
     2^{\frac{p+1}{p}}\sigma\alpha^{\frac{p-1}{p}} \leq \frac{\epsilon}{3}.
\end{align}
Additionally, we have
\begin{align*}
    R = \left\lceil \frac{1}{\alpha}\right\rceil
    \leq \frac{1}{\alpha} + 1,
\end{align*}
which implies
\begin{align}\label{eq:avg_grad_bound_2}
    2L\eta\Big(\frac{1}{\alpha}+R+1\Big)
\le 2L\eta\Big(\frac{1}{\alpha}+\frac{1}{\alpha}+1+1\Big)    
\le 2L\cdot\frac{\alpha\epsilon}{24L}\cdot\frac{4}{\alpha}
= \frac{\epsilon}{3}.
\end{align}
Combining equations (\ref{eq:avg_grad_bound_refined}), (\ref{eq:avg_grad_bound_1}), and  (\ref{eq:avg_grad_bound_2}) yields
\begin{align*}
   \E\|\nabla f(\hat x)\| = \frac{1}{K}\sum_{k=0}^{K-1}\E\|\nabla f(x_k)\| \leq  \epsilon.
\end{align*}
\end{proof}

\subsection{The Proof of Corollary \ref{cor:time_mom_refined}}
We first introduce the following lemma for the strategy of delay threshold for the fixed computation model~\citep{pmlr-v267-maranjyan25b}.
\begin{lem}[{\citet[Lemma 4.1]{pmlr-v267-maranjyan25b}}]
\label{lem:ringmaster_time}
Under Assumption~\ref{ass:time}, let $R\ge 1$ be the delay threshold, and let $t(R)$ be the wall-clock time required to complete any $R$ consecutive accepted updates. Then it holds
\begin{equation}
\label{eq:ringmaster_lemma}
t(R) \le 2 \min_{m\in[n]} \left[\left( \frac{1}{m}\sum_{i=1}^m \frac{1}{\tau_i} \right)^{-1} \left( 1 + \frac{R}{m} \right)\right].
\end{equation}
\end{lem}

We now provide the proof of Corollary \ref{cor:time_mom_refined}.

\begin{proof}

Without loss of generality, we focus on the small $\epsilon>0$ such that
\begin{equation}\label{eq:eps_nontrivial}
\epsilon \ \le\ \sqrt{2L\Delta}.
\end{equation}
Otherwise, Assumption \ref{ass:smooth} implies
\begin{align*}
\begin{split}    
f^* \leq &  f\left(x_0-\frac{1}{L}\nabla f(x_0)\right) \\
\leq & f(x_0) - \inner{\nabla f(x_0)}{\frac{1}{L}\nabla f(x_0)} + \frac{L}{2}\Norm{\frac{1}{L}\nabla f(x_0)}^2 \\
\leq & f(x_0) - \frac{1}{2L}\Norm{\nabla f(x_0)}^2.
\end{split}
\end{align*}
Rearranging above inequality by taking $\epsilon>\sqrt{2L\Delta}$ leads to
\begin{align*}
\Norm{\nabla f(x_0)} \leq \sqrt{2L(f(x_0) - f^*)} = \sqrt{2L\Delta} < \epsilon,
\end{align*}
which means $x_0$ is already an $\epsilon$-stationary point and the claim is trivial.
Hence, we assume $\epsilon \ \le\ \sqrt{2L\Delta}$ throughout the remainder of the proof.

According to Theorem~\ref{thm:main_mom_tight}, we have
\begin{align*}
   \E\|\nabla f(\hat x)\| \leq  \epsilon.
\end{align*}
Hence, we only need to consider the time complexity of the algorithm.
According to Lemma~\ref{lem:ringmaster_time},
the time complexity required to complete any $R$ consecutive  accepted updates $t(R)$ has the upper bound
\begin{equation}
\label{eq:t1R_refined}
t(R)\ \le\ 2\min_{m\in[n]}\left[\left(\frac{1}{m}\sum_{i=1}^m\frac{1}{\tau_i}\right)^{-1}\left(1+\frac{R}{m}\right)\right].
\end{equation}
By partitioning all $K$ accepted updates into blocks with length $R$, the overall time complexity $T_p(\epsilon)$ of the algorithm is bounded by
\begin{align}\label{eq:T_expansion_strict}
\begin{split}    
 T_p(\epsilon)
\le & t(R)\left\lceil\frac{K}{R}\right\rceil \\
\le &
2\min_{m\in[n]}
\left(\frac{1}{m}\sum_{i=1}^m\frac{1}{\tau_i}\right)^{-1}
\left(1+\frac{R}{m}\right)\left(\frac{K}{R}+1\right)\\
= &
2\min_{m\in[n]}
\left(\frac{1}{m}\sum_{i=1}^m\frac{1}{\tau_i}\right)^{-1}
\left(\frac{K}{m}+\frac{R}{m}+\frac{K}{R}+1\right).
\end{split}
\end{align}
where the second step is based on 
the fact $\lceil x\rceil\le x+1$ and equation (\ref{eq:t1R_refined}).

According to \eqref{eq:eps_nontrivial} and the setting $\alpha \le 1$, we have
\begin{align}\label{eq:Kbound}
K  =  \left\lceil \frac{72L\Delta}{\alpha\epsilon^2}\right\rceil
\le \frac{72L\Delta}{\alpha\epsilon^2}+1
\le \left(1+\frac{1}{36}\right)\frac{72L\Delta}{\alpha\epsilon^2}
\ =\ \frac{74 L\Delta}{\alpha\epsilon^2}.
\end{align}
Consequently, we have
\begin{align}
\label{eq:Km_new}
\frac{K}{m}\ \le\ \frac{74 L\Delta}{m\alpha\epsilon^2}.
\end{align}

Moreover, the parameter settings imply $\alpha\le 1$ and $R\le \lceil 1/\alpha\rceil\le 1/\alpha+1\le 2/\alpha$, which means
\begin{align}\label{eq:Rm_new}
\frac{R}{m} \le \frac{2}{m\alpha} \le \frac{4L\Delta}{m\alpha\epsilon^2}.
\end{align}
where the last inequality is based on \eqref{eq:eps_nontrivial}.

According to \eqref{eq:Kbound}, we have
\begin{align}\label{eq:KR_new}
\frac{K}{R}
\ \le\ \frac{74L\Delta}{\alpha\epsilon^2}\cdot\frac{1}{R}
\leq \frac{74L\Delta}{\alpha\epsilon^2}\cdot \alpha
\leq \frac{74L\Delta}{\epsilon^2},
\end{align}
where the second inequality is based on the setting $R=\lceil 1/\alpha \rceil \ge 1/\alpha $.

Combining equations (\ref{eq:eps_nontrivial}), (\ref{eq:T_expansion_strict}), (\ref{eq:Km_new}), (\ref{eq:Rm_new}), (\ref{eq:KR_new}) and the setting of $\alpha$ in Theorem \ref{thm:main_mom_tight}, 
we obtain
\begin{align*}    
\begin{split}    
T_p(\epsilon) &\leq  2\min_{m\in[n]}
\left(\frac{1}{m}\sum_{i=1}^m\frac{1}{\tau_i}\right)^{-1}
\left(\frac{74L\Delta}{m\alpha\epsilon^2}  + \frac{4L\Delta}{m\alpha\epsilon^2}  + \frac{74L\Delta}{\epsilon^2} + 1\right) \\
&\leq  2\min_{m\in[n]}
\left(\frac{1}{m}\sum_{i=1}^m\frac{1}{\tau_i}\right)^{-1}
\left(\frac{74L\Delta}{m\alpha\epsilon^2}  + \frac{4L\Delta}{m\alpha\epsilon^2}  + \frac{74L\Delta}{\epsilon^2} + \frac{2L\Delta}{\epsilon^2}\right) \\
& =  2\min_{m\in[n]}
\left(\frac{1}{m}\sum_{i=1}^m\frac{1}{\tau_i}\right)^{-1}
\left(\frac{78L\Delta}{m\alpha\epsilon^2} + \frac{76L\Delta}{\epsilon^2}\right) \\
&=  \fO\left(\min_{m\in[n]}
\left(\frac{1}{m}\sum_{i=1}^m\frac{1}{\tau_i}\right)^{-1}
\left(
\frac{L\Delta}{\epsilon^{2}}
+
\frac{L\Delta}{m\epsilon^{2}}
\left(\frac{\sigma}{\epsilon}\right)^{\frac{p}{p-1}}
\right)\right).
\end{split}
\end{align*}
\end{proof}

\subsection{The Proof of Corollary \ref{lem:universal_time_recursion_blocks}}

We first introduce the following lemma for the strategy of delay threshold for the universal computation model.
\begin{lem}[{\citet[Lemma 5.1]{pmlr-v267-maranjyan25b}}]
\label{lem:universal_time}
Under Assumption~\ref{ass:universalmodel}, we let $R$ be the delay threshold in Theorem \ref{thm:main_mom_tight}. Assume that some iteration starts at time $T_0$, then the $R$ consecutive iterative updates of Algorithm \ref{alg:ringmaster_nsgd_mom_server} starting from $T_0$ will be performed before the time
\begin{align*}
T(R,T_0):=
\min\left\{
T \ge 0: \frac{1}{4}\sum_{i=1}^{n}\int_{T_0}^{T} \upsilon_i(\tau)\,{\rm d}\tau \ge R
\right\}.
\end{align*}
\end{lem}

We now provide the proof of Corollary \ref{lem:universal_time_recursion_blocks}.
\begin{proof}
We follow the parameter settings in Theorem~\ref{thm:main_mom_tight} and focus on the small $\epsilon>0$ such that $\epsilon \le \sqrt{2L\Delta}$. 
Following the derivation of \eqref{eq:KR_new}, we have
\begin{align*}
K \le R \times \left\lceil \frac{K}{R}\right\rceil
\le
R \times \frac{74L\Delta}{\epsilon^2}
\end{align*}
and the algorithm can find an $\epsilon$-stationary point with at most $K$ rounds of updates on the server.

We define
\begin{align*}
\bar K := \left\lceil\frac{74L\Delta}{\epsilon^2} \right\rceil,
\end{align*}
then we can partition all $K$ rounds of updates on the server into $\bar K$ epochs, where each epoch contains at most $R$ consecutive rounds of updates.

According to Lemma \ref{lem:universal_time}, we know that Algorithm \ref{alg:ringmaster_nsgd_mom_server} can finish the first epoch of updates on the server before the time
\begin{align*}
T_1 := T(R,0)
= \min\left\{T\ge 0: \frac{1}{4}\sum_{i=1}^n \int_{0}^{T} \upsilon_i(\tau) \,{\rm d} \tau\ge R\right\}.
\end{align*}
Since the algorithm will finish the first epoch at most $T_1$ seconds, it will start to perform the second epoch before the time~$T_1$. 
Again using Lemma \ref{lem:universal_time}, Algorithm \ref{alg:ringmaster_nsgd_mom_server} can finish the second epoch before the time
\begin{align*}
T_2 := T(R,T_1)
= \min\left\{T\ge 0: \frac{1}{4}\sum_{i=1}^n \int_{T_1}^{T} \upsilon_i(\tau)\,{\rm d} \tau\ge R\right\}.
\end{align*}
By analogy, Algorithm \ref{alg:ringmaster_nsgd_mom_server} can finish all the $\bar K$ epochs before the time 
\begin{align*}
T_{\bar K} := T(R,T_{\bar K-1})
= \min\left\{T\ge 0: \sum_{i=1}^n \left[\frac{1}{4}\int_{T_{\bar K-1}}^{T} \upsilon_i(\tau)\,{\rm d} \tau\right]\ge R\right\},
\end{align*}
which finishes the proof.
\end{proof}

\section{The Proofs for Results in Section \ref{sec:lower-bound}}\label{sec:lb_asgd_heavytail}

In this section, we construct the hard instances to provide the lower complexity bounds of asynchronous stochastic first-order oracle algorithm for both fixed and universal models.
Specifically, we extend the zero-chain hard instance \cite{arjevani2023lower} and the asynchronous progress bound \cite{tyurin2023optimal} to the stochastic first-order oracle under the $p$-BCM condition (Assumption \ref{ass:pBCM}).

\subsection{Preliminaries for Lower Bound Analysis}

We first introduce the notation for the progress index.

\begin{dfn}[Progress index]
\label{dfn:prog_index}
For a vector $x=[x_{(1)},\dots,x_{(d)}]^\top\in\R^{d}$ and threshold $\alpha\ge 0$, we define
\begin{align*}
\operatorname{prog}_{\alpha}(x)
:=
\max\bigl\{i\in \{0,1,\dots,d\} : |x_{(i)}|>\alpha\bigr\}.
\end{align*}
We also denote $x_{(0)}\equiv 1$, $\operatorname{prog}_{\alpha}(0)=0$, and $\operatorname{prog}(x):=\operatorname{prog}_0(x)$ for convention.
\end{dfn}

We then introduce the zeroth-chain function as follows \citep{arjevani2023lower}.
\begin{lem}[{\citet[Lemma 2]{arjevani2023lower}}]
\label{lem:hard_chain}
For given integer $d\geq 1$ and $x=[x_{(1)},\dots,x_{(d)}]\in\R^d$, we define
\begin{equation}
\label{eq:HT_def}
H_d(x)
:=
-\Psi(1) \Phi(x_{(1)})
+
\sum_{i=2}^{d}\Bigl(\Psi(-x_{(i-1)}) \Phi(-x_{(i)})-\Psi(x_{(i-1)}) \Phi(x_{(i)})\Bigr),
\end{equation}
where 
\begin{equation}
\label{eq:Psi_def}
\Psi(t)
:=
\begin{cases}
0, & t \le \dfrac12, \\[0.35cm]
\exp\left(1-\dfrac{1}{(2t-1)^2}\right), & t > \dfrac12,
\end{cases}
\qquad\text{and}\qquad
\Phi(t)
:=
\sqrt{{\rm e}}\displaystyle\int_{-\infty}^{t} \exp\left(-\frac{\tau^2}{2} \right)\, {\rm d}\tau.
\end{equation}
Then we have
\begin{enumerate}[label=(\alph*)]
    \item $H_d(0) - \inf_{x\in\BR^d} H_d(x) \le 12 d$;
    \item $H_d$ is $152$-smooth;
    \item for all $x \in \R^d$, it holds $\Norm{\nabla H_d(x)}_\infty \le 23$;
    \item for all $x \in \R^d$, it holds
    \begin{equation}
    \label{eq:prog_control}
    \operatorname{prog}_0 (\nabla H_d(x)) \ \le \ \operatorname{prog}_{1/2}(x) + 1;
    \end{equation}
    \item if $\operatorname{prog}_1 (x) < d$, then
    \begin{equation}
    \label{eq:grad_lb_hard}
    \Norm{\nabla H_d(x)} \ \ge \ 1.
    \end{equation}
\end{enumerate}
\end{lem}

Note that the zero-chain instance enforces the gradient on coordinate $j$ be zero unless the previous coordinate $j-1$ has already been sufficiently activated. 
In the stochastic setting, we would like to amplify the hardness by placing a \emph{probabilistic gate} on the \emph{newly-revealed} coordinate of stochastic gradient, i.e., even when the next coordinate becomes admissible, the oracle reveals its gradient signal only with probability $q$
(and otherwise returns $0$ on that coordinate). 
For the $p$-BCM gradient noise (Assumption~\ref{ass:pBCM}),
we would like to design the appropriate probabilistic gate to construct the stochastic first-order oracle.

For the fixed computation model shown in Definition \ref{dfn:time_multi_oracle}, at round $k$, the algorithm perform~$A_k$ to select a wall-clock time~$t_{k+1}$ such that $t_{k+1}\ge t_k$ and worker $i_{k+1}$ to output
\begin{align*}    
(t_{k+1}, i_{k+1}, x_k)\ :=\ A_k(g_1,\dots,g_k),
\end{align*}
where $x_k\in\BR^d$ satisfies
\begin{align*}
\mathrm{supp}(x_k) \ \subseteq \ \bigcup_{s \in [k]} \mathrm{supp}(g_s)    
\end{align*}
and $\{g_1,\dots,g_k\}$ are observed gradient estimates  from the outputs of the past oracle calls.
Intuitively, the progress index set~$\operatorname{prog}(x_k)$ can increase only after the algorithm has observed a gradient with a newly nonzero coordinate beyond the current support set.

We then formally define the algorithm class for the universal computation model as follows.

\begin{dfn}\label{dfn:time_multi_oracle_universal}
An asynchronous stochastic first-order
oracle algorithm $\fA=\{A_k\}_{k\geq 0}$ over $n$ workers under the universal computation model satisfies the following constraints.

\begin{itemize}[leftmargin=0.3cm,topsep=-0.1cm,itemsep=-0.1cm]
\item Each worker $i\in[n]$ maintains an internal state
\[
s_i=(s_{t,i}, s_{x,i}, s_{q,i})\in \R_{\ge 0}\times \R^{d} \times \{0,1\},
\]
where $s_{t,i}\in\BR$ stores  start time of 
current computation,
$s_{x,i}\in\BR^{d}$ stores the corresponding requested point,
and $s_{q,i}$ stores the state of worker $i$, i.e., $s_{q,i}=0$ when worker $i$ is idle and $s_{q,i}=1$ when worker $i$ is busy.
\item Each worker $i\in[n]$ is associated with the time sensitive stochastic first-order oracle $O^{\widehat{\nabla} F}_{V_i}$, which takes input as the current time $t$, requested point $x$, and current state~$s_i$ and output  $(s_{i,+},g)=O^{\widehat{\nabla} F}_{V_i}$ such that 
\begin{align*}
\bigl(s_i^{+}, g\bigr)=
O^{\widehat{\nabla} F}_{V_i}(t,x,s_i,\xi):=
\begin{cases}
\bigl((t,x,1),0\bigr),
& s_{q,i}=0,\\[2pt]
\bigl((s_{t,i},s_{x,i},1),0\bigr),
& s_{q,i}=1 \ \text{and}\ V_i(t)-V_i(s_{t,i})<1,\\[2pt]
\bigl((0,0,0),\widehat{\nabla} F(s_{x,i};\xi)\bigr),
& s_{q,i}=1 \ \text{and}\ V_i(t)-V_i(s_{t,i})\ge 1.
\end{cases}
\end{align*}
where $\xi$ is sampled from distribution~$\fD$ and $\widehat{\nabla} F(s_{x,i};\xi)$ can be accessed by the universal computation model described by Assumption \ref{ass:universalmodel} and satisfies the conditions
$\mathbb{E} [ \widehat{\nabla} F(s_{x,i}; \xi)] = \nabla f(s_{x,i})$ and 
$\mathbb{E}[\|\widehat{\nabla} F(s_{x,i}; \xi) - \nabla f(s_{x,i}) \|^p] \le \sigma^p$ for some $p\in(1,2]$ and~$\sigma>0$. 
This implies the stochastic gradient computation started at time $s_{t,i}$ is completed by time~$t\ge s_{t,i}$ once $V_i(t)\ge V_i(s_{t,i})+1$.
\item The asynchronous stochastic first-order
oracle algorithm $\fA=\{A_k\}_{k\ge 0}$ initializes $t_0=0$ and $x_0\in\R^{d}$, then performs oracles $\{O^{\widehat{\nabla} F}_{\tau_i}\}_{i=1}^n$ 
during the iterations. 
At round $k$, after observing gradient estimates $\{g_1,\dots,g_k\}$ from the outputs of the past oracle calls, 
the algorithm perform~$A_k$ to select a wall-clock time $t_{k+1}$ such that $t_{k+1}\ge t_k$ and worker $i_{k+1}$ to output.
\end{itemize}
\end{dfn}

For any wall-clock time $t\ge 0$, define the time-indexed set of available iterates
\[
\fS_t := \bigl\{ k\in\BN_0 : t_k \le t\bigr\}.
\]
for both the fixed computation model (Definition \ref{dfn:time_multi_oracle}) and the universal computation model (Definition \ref{dfn:time_multi_oracle_universal}).

\subsection{The Proof of Theorem \ref{thm:lb_asgd_pBCM}}\label{appendix:lower-fixed}

We first introduce the following lemma for the fixed computation model \cite{tyurin2023optimal}.

\begin{lem}[{\citet[Lemma D.2]{tyurin2023optimal}}]
\label{lem:async_progress}
Let $\fA$ be any algorithm that follows Definition~\ref{dfn:time_multi_oracle} associated with a differentiable function $f:\R^{d}\to\R$ such that
\[
\operatorname{prog}(\nabla f(x)) \ \le\ \operatorname{prog}(x)+1
\]
for all $x\in\BR^d$
and the stochastic first-order oracle  $\widehat{\nabla} F$ such that 
\begin{equation}
\label{eq:tyurin_mapping_21}
\left(\widehat{\nabla} F(x;\xi)\right)_{(j)}
\ =\
\left(\nabla f(x)\left(1+\mathbf{1}\!\left[j>\operatorname{prog}(x)\right]\left(\frac{\xi}{q}-1\right)\right)\right)_{(j)},
\end{equation}
for all $x\in\R^{d}$, $\xi\in\{0,1\}$, and $j\in[d]$, where we take $\xi\in\mathrm{Bernoulli}(q)$ for some $q\in(0,1]$.
Then with probability at least~$1-\delta$, it holds
\[
\inf_{k\in \fS_t}\mathbf{1}\!\left[\operatorname{prog}(x_k)<d\right]\ \ge\ 1
\] 
for all wall-clock times $t$ such that
\begin{equation}
\label{eq:async_time_bound}
t \ \le\ \frac{1}{24}\min_{m\in[n]}
\left[
\left(\sum_{i=1}^m \frac{1}{\tau_i}\right)^{-1}
\left(\frac{1}{q}+m\right)
\right]
\left(\frac{d}{2}+\log\delta\right),
\end{equation}
where $\delta\in(0,1)$, $\vone[\cdot]$ is the indicator function, and $\fS_t \;:=\; \bigl\{ k\in\BN_0 : t_k \le t\bigr\}$.
\end{lem}

We now provide the proof of Theorem \ref{thm:lb_asgd_pBCM}.
\begin{proof}

We follow the function $H_d:\BR^d\to\BR$ defined in Lemma~\ref{lem:hard_chain} to construct the instance
\begin{equation}
\label{eq:lb_scaled_f_explicit_final}
f(x)\ :=\ \frac{L\lambda^2}{152}\ H_d\left(\frac{x}{\lambda}\right),
\end{equation}
where $\lambda>0$.
We then complete the proof by considering the function class, the oracle class, and the time complexity.

\paragraph{Part I: The Function Class.} 
We verify the properties of $f$ as follows.
\begin{enumerate}
\item We take
\begin{equation}
\label{eq:lb_d_choice_explicit_final}
d:= \floor{\frac{38\Delta}{3L\lambda^2}},
\end{equation}
then it holds  
\[
f(0)-\inf_{x\in\BR^d} f(x)
=\frac{L\lambda^2}{152}\bigl(H_d(0)-\inf_{x\in\BR^d} H_d(x)\bigr)
\le \frac{L\lambda^2}{152}\cdot 12d \leq \Delta,
\]
where the first inequality is based on Lemma~\ref{lem:hard_chain}(a)
\item For all $x,y\in\BR^d$, we have
\[
\Norm{\nabla f(x)-\nabla f(y)}
=\frac{L\lambda}{152}\Norm{\nabla H_d\left(\frac{x}{\lambda}\right)-\nabla H_d\left(\frac{y}{\lambda}\right)}
\le L\Norm{x-y},
\]
where the inequality is based on Lemma~\ref{lem:hard_chain}(b), i.e., the function $H_d$ is 152-smooth, 
Hence, the function $f$ is $L$-smooth. 
\end{enumerate}

\paragraph{Part II: The Oracle Class.} We then follow Lemma \ref{lem:async_progress} to construct the stochastic gradient estimator $\widehat{\nabla} F(x;\xi)$ for the objective $f(x)$.
Specifically, we set $\fD$ as Bernoulli distribution with parameter $q\in(0,1]$ and define each coordinate of $\widehat{\nabla} F(x;\xi)$ as
\begin{align*}
\left(\widehat{\nabla} F(x;\xi)\right)_{(j)}
=
\left(\nabla f(x)\left(1+\mathbf{1} \left[j>\operatorname{prog}(x)\right]\left(\frac{\xi}{q}-1\right)\right)\right)_{(j)},
\end{align*}
where $\xi\sim {\rm Bernoulli}(q)$, $j\in[d]$, and $\mathbf{1}[\cdot]$ is the indicator function.
We now verify the stochastic gradient estimator $\widehat{\nabla} F(x;\xi)$ is unbiased and satisfies the $p$-BCM assumption.

\begin{enumerate}
\item The setting $\xi\sim {\rm Bernoulli}(q)$ implies $\BE[\xi]=q$, which means
\begin{align*}
\BE\left[\left(\widehat{\nabla} F(x;\xi)\right)_{(j)}\right]
= \BE\left[\left(\nabla f(x)\left(1+\mathbf{1} \left[j>\operatorname{prog}(x)\right]\left(\frac{\xi}{q}-1\right)\right)\right)_{(j)}\right]
= \nabla f(x)_{(j)}
\end{align*}
for all $x\in\BR^d$ and $j\in[d]$.
Hence, stochastic the gradient estimator $\widehat{\nabla} F(x;\xi)$ is unbiased. 
\item According to Lemma~\ref{lem:hard_chain}(d), we have
\[
\operatorname{prog}_0(\nabla H_d(u))\le \operatorname{prog}_{1/2}(u)+1
\le \operatorname{prog}_0(u)+1
\]
for all $u\in\R^d$.
This implies
\[
\operatorname{prog}(\nabla f(x))\le \operatorname{prog}(x)+1,
\]
since $f(x)=(L\lambda^2/152)H_d\left({x}/{\lambda}\right)$ and $\operatorname{prog}(\cdot)=\operatorname{prog}_0(\cdot)$.
This implies the entries of 
\begin{align*}
\zeta(x;\xi):=\widehat{\nabla} F(x;\xi) - \nabla f(x)
\end{align*}
can be nonzero only on coordinates such that $j \leq \operatorname{prog}(x) + 1$.
On the other hand, the term $\mathbf{1} \left[j>\operatorname{prog}(x)\right]$ in above equation  implies the entries of 
$\zeta(x;\xi)$ can be nonzero only on coordinate such that $j>\operatorname{prog}(x)$.
Therefore, there is at most one entry of $\zeta(x;\xi)$ can be non-zeroth, which leads to
\begin{align}\label{eq:lower-bcm}
 \E\left[\Norm{\widehat{\nabla} F(x;\xi)-\nabla f(x)}^p\right]  
\leq    \Norm{\nabla f(x)}_\infty^p\BE\left[\left|\frac{\xi}{q}-1\right|^p\right].
\end{align}
For the term $\Norm{\nabla f(x)}_{\infty}$, we have
\begin{align}\label{eq:lower-bcm1}
\Norm{\nabla f(x)}_\infty
=\frac{L\lambda}{152}\Norm{\nabla H_d(x/\lambda)}_\infty
\le \frac{L\lambda}{152}\cdot 23
=\frac{23L\lambda}{152}    
\end{align}
for all $x\in\BR^d$, where the inequality is based on Lemma~\ref{lem:hard_chain}(c). 

For the term $\BE[|\xi/q-1|^p]$, the settings of $\xi\sim\mathrm{Bernoulli}(q)$ and $p\in(1,2]$ imply
\begin{align}\label{eq:lower-bcm2}
\E\left[\left|\frac{\xi}{q}-1\right|^p\right]
= q\left(\frac{1-q}{q}\right)^p+(1-q)
\le q^{1-p}+1
\le \frac{2}{q^{p-1}}.
\end{align}

Combining equations (\ref{eq:lower-bcm}), (\ref{eq:lower-bcm1}), and (\ref{eq:lower-bcm2}), we have
\begin{equation}
\label{eq:lb_pBCM_check_explicit_final}
\E\left[\Norm{\widehat{\nabla} F(x;\xi)-\nabla f(x)}^p\right]
\le \left(\frac{23L\lambda}{152}\right)^p\cdot \frac{2}{q^{p-1}}.
\end{equation}
Hence, taking
\begin{equation}
\label{eq:lb_q_choice_explicit_final}
q
\ :=\
\min\left\{
\left(\frac{2^{\frac1p}\cdot 23L\lambda}{152\sigma}\right)^{\frac{p}{p-1}},1
\right\},
\end{equation}
for equation \eqref{eq:lb_pBCM_check_explicit_final} guarantees 
\begin{align*}
    \E\left[\Norm{\widehat{\nabla} F(x;\xi)-\nabla f(x)}^p\right]  
\leq  \sigma^p
\end{align*}
for all $x\in\BR^d$.
\end{enumerate}

\paragraph{Part III: The Time Complexity.}
We now fix
\begin{equation}
\label{eq:lb_lambda_choice_explicit_final}
\lambda\ := \frac{608\eps}{L}.
\end{equation}
Combining with \eqref{eq:lb_d_choice_explicit_final}, we obtain 
\begin{align}\label{eq:dim}
d=\left\lfloor\frac{38\Delta}{3L\lambda^2}\right\rfloor
=\left\lfloor\frac{L\Delta}{29184\eps^2}\right\rfloor.  
\end{align}

We then show that 
\begin{equation}
\label{eq:lb_grad_indicator_explicit_final}
\Norm{\nabla f(x)}
\ \ge\ 4\eps\cdot \mathbf{1}\!\left[\operatorname{prog}(x)<d\right].
\end{equation}
We only need to consider the case of $\operatorname{prog}(x)<d$, which implies  $\operatorname{prog}_0(x/\lambda)=\operatorname{prog}(x)<d$. Therefore, it holds
\begin{align*}
\operatorname{prog}_1\left(\frac{x}{\lambda}\right) \le \operatorname{prog}_0\left(\frac{x}{\lambda}\right)<d.    
\end{align*}
Based on definition in equation (\ref{eq:lb_scaled_f_explicit_final}), we have
\[
\Norm{\nabla f(x)}
=\frac{L\lambda}{152}\Norm{\nabla H_d\left(\frac{x}{\lambda}\right)}
\ge \frac{L\lambda}{152}
=\frac{L}{152}\cdot \frac{608\eps}{L}
=4\eps,
\]
where the inequality is based on Lemma \ref{lem:hard_chain} (e) since $\operatorname{prog}(x)<d$.

Therefore, we have
\begin{equation}
\label{eq:lb_bridge_explicit_final}
\inf_{k\in \fS_t}\Norm{\nabla f(x_k)}
\ \ge\
4\eps\cdot \inf_{k\in \fS_t}\mathbf{1} \left[\operatorname{prog}(x_k)<d\right],
\end{equation}
where $\fS_t=\{k:\ t_k\le t\}$ collects all indices on the server whose query times do not exceed $t$

According to Lemma~\ref{lem:async_progress} with $\delta=1/2$,  with probability at least $1/2$, it holds
\begin{align}\label{eq:progxkd1}    
\inf_{k\in \fS_t}\mathbf{1} \left[\operatorname{prog}(x_k)<d\right]\ \ge\ 1
\end{align}
for all $t$ satisfying
\begin{align}\label{eq:t24}
t \ \le\ \frac{1}{24}\min_{m\in[n]}
\left[
\left(\sum_{i=1}^m \frac{1}{\tau_i}\right)^{-1}
\left(\frac{1}{q}+m\right)
\right]
\left(\frac{d}{2}+\log\frac{1}{2}\right).    
\end{align}
Combining equations (\ref{eq:lb_bridge_explicit_final}) and (\ref{eq:progxkd1}) means for all $t$ satisfying \eqref{eq:t24} holds
\[
\E\!\left[\inf_{k\in \fS_t}\Norm{\nabla f(x_k)}\right]
\ \ge\
4\eps\cdot \Pr\!\left(\inf_{k\in \fS_t}\mathbf{1}[\operatorname{prog}(x_k)<d]\ge 1\right)
\ \ge\ 2\eps.
\]

Hence, if we can find an $\epsilon$-stationary point in expectation by query times do not exceed $t$, i.e., $\E[\inf_{k\in \fS_t}\|\nabla f(x_k)\|]\le \eps$, then the time $t$ must violate \eqref{eq:t24}, i.e.,
\begin{align}\label{eq:t24b}
t > \frac{1}{24}\min_{m\in[n]}
\left[
\left(\sum_{i=1}^m \frac{1}{\tau_i}\right)^{-1}
\left(\frac{1}{q}+m\right)
\right]
\left(\frac{d}{2}+\log\frac{1}{2}\right).    
\end{align}
Note that for all $\epsilon\le \sqrt{{L\Delta}/{87552}}=\fO(\sqrt{L\Delta})$, equation (\ref{eq:dim}) implies
\begin{align}\label{eq:final-d}
d=\left\lfloor\frac{L\Delta}{29184\eps^2}\right\rfloor \geq 3.  
\end{align}
Note that combining equations (\ref{eq:lb_q_choice_explicit_final}) and (\ref{eq:lb_lambda_choice_explicit_final}) achives
\begin{align}\label{eq:final-q}
q = \min\left\{
\left(\frac{2^{\frac1p}\cdot 23L}{152\sigma}\cdot \frac{608\eps}{L}\right)^{\frac{p}{p-1}},1
\right\}
= \min\left\{
\left(\frac{92\cdot 2^{\frac1p}\eps}{\sigma}\right)^{\frac{p}{p-1}},1
\right\}.
\end{align}
Combining equations (\ref{eq:t24b}), (\ref{eq:final-d}), and (\ref{eq:final-q}) implies finding an $\epsilon$-stationary point in expectation requires the time complexity at least 
\begin{align*}
t > \frac{1}{24}\min_{m\in[n]}
\left[
\left(\sum_{i=1}^m \frac{1}{\tau_i}\right)^{-1}
\left(\frac{1}{q}+m\right)
\right]
\left(\frac{d}{2}+\log\frac{1}{2}\right)
= \Omega\left(\min_{m\in[n]}
\left(\frac{1}{m}\sum_{i=1}^m \frac{1}{\tau_i}\right)^{-1}
\left(
\frac{L\Delta}{\eps^{2}}
+
\frac{\sigma^{\frac{p}{p-1}} L\Delta}{m \cdot \eps^{\frac{3p-2}{p-1}}}
\right)\right).
\end{align*}
\end{proof}

\subsection{The Proof of Theorem \ref{thm:lb_asgd_pBCM}}
\label{sec:lb_asgd_heavytail_universal}

This subsection provides the lower bound for the universal computation model.
The main ideas of our construction are similar to the ideas for the fixed computation model in Section \ref{appendix:lower-fixed}, i.e., we also use the function instance $H_d$ defined in Lemma \ref{lem:hard_chain} \citep{arjevani2023lower} and the Bernoulli-gated oracle construction (Part II in the proof of Theorem \ref{thm:lb_asgd_pBCM}) to construct the hard instance.
The main difference is that we replace the fixed computation model (Assumption \ref{ass:time})
with the universal computation model (Assumption \ref{ass:universalmodel}) \citep{tyurin2025tight}, which allows time-varying computation dynamics and leads to the different progress-to-time conversion.

Recall that Assumption \ref{ass:universalmodel} suppose the number of stochastic gradients can be calculated from a time $t_0$ to a time $t_1$ at worker $i\in[n]$ is the Riemann integral of the computation power followed by the floor operation, i.e., 
\begin{align*}
\left\lfloor \int_{t_0}^{t_1} \upsilon_i(\tau)\,{\rm d}\tau \right\rfloor=\lfloor V_i(t_1)-V_i(t_0)\rfloor,
\end{align*}
where $\upsilon_i:\BR_+\to\BR_+$ is the computation power of worker $i$ which is continuous almost everywhere
and
\begin{align*}
 V_i(t):=\int_{0}^{t} \upsilon_i(\tau)\,{\rm d}\tau
 \end{align*}
defined on $t\in\R_{\ge 0}$ is the cumulative computation work function of worker $i$.

We first present the following two auxiliary lemmas (Appendix~G of \citet{tyurin2025tight}).

\begin{lem}[{\citet[Lemma~G.1]{tyurin2025tight}}]
\label{lem:tyurin25_G1}
Let $V_i:\R_{+}^{\infty}\to \R_{+}^{\infty}$ be continuous and non-decreasing function for all $i\in[n]$.
For all $\kappa, b_1,\dots,b_n\in \R_{+}^{\infty}$, the minima of the sets
\[
\Bigl\{t\ge 0:\ \sum_{i=1}^{n}\lfloor V_i(t)-b_i\rfloor \ge \kappa\Bigr\}
\quad\text{and}\quad
\left\{t\ge 0:\ \left(\frac{1}{n}\sum_{i=1}^{n}\frac{1}{\lfloor V_i(t)-b_i\rfloor}\right)^{-1}\ge \kappa\right\}
\]
exist, where we using the convention $\min\{\emptyset\}=+\infty$.
\end{lem}

\begin{lem}[{\citet[Lemma~G.2]{tyurin2025tight}}]
\label{lem:tyurin25_G2}
Let $T\ge 1$ and $\{\vartheta_i\}_{i=1}^{T}$ be geometric random variables such that for given $\vartheta_1,\dots,\vartheta_{i-1}$, the random variable $\vartheta_{i}$ follows the geometric distribution
\[
\vartheta_i \sim \mathrm{Geometric} \, \bigl(\pi_{i,\vartheta_1,\dots,\vartheta_{i-1}}\bigr)
\]
where the parameter $\pi_{i,\vartheta_1,\dots,\vartheta_{i-1}}\in(0,1]$
depends only on $\vartheta_1,\dots,\vartheta_{i-1}$ for all $i\in[T]$.
Then we have
\[
\Pr\left(
\sum_{i=1}^{T}\mathbf{1}\left[\vartheta_i>\frac{1}{4\pi_{i,\vartheta_1,\dots,\vartheta_{i-1}}}\right]
\le \frac{T}{2}+\log \delta
\right)
\le \delta
\]
for all $\delta\in(0,1]$.
\end{lem}

We now provide the universal dynamics counterpart of Lemma~\ref{lem:async_progress} as follows.
\begin{lem}\label{lem:async_progress_universal}
Let $\fA$ be any algorithm that follows Definition~\ref{dfn:time_multi_oracle_universal} associated with a differentiable function $f:\R^{d}\to\R$ such~that
\[
\operatorname{prog}(\nabla f(x)) \ \le\ \operatorname{prog}(x)+1
\]
for all $x\in\BR^d$
and the stochastic first-order oracle  $\widehat{\nabla} F$ such that 
\begin{equation}
\label{eq:tyurin_mapping_212}
\left(\widehat{\nabla} F(x;\xi)\right)_{(j)}
\ =\
\left(\nabla f(x)\left(1+\mathbf{1}\!\left[j>\operatorname{prog}(x)\right]\left(\frac{\xi}{q}-1\right)\right)\right)_{(j)},
\end{equation}
for all $x\in\R^{d}$, $\xi\in\{0,1\}$, and $j\in[d]$, where we take $\xi\sim\mathrm{Bernoulli}(q)$ for some $q\in(0,1]$.
Then with probability at least~$1-\delta$, it holds
\[
\inf_{k\in \fS_t}\mathbf{1}\!\left[\operatorname{prog}(x_k)<d\right]\ \ge\ 1
\] 
for all wall-clock times $t$ such that  $t\le T_{\left\lfloor d/2+\log\delta\right\rfloor}$, where $\delta\in(0,1)$, $\fS_t=\bigl\{k\in\BN_0:t_k \le t\bigr\}$, and $\vone[\cdot]$ is the indicator function.
Here, the we recursively defined sequence $\{T_K\}_{K\ge 0}$ such that $T_0= 0$ and
\begin{align}\label{eq:TK_universal_recursion_aligned}
    T_K &:= \min\left\{ T \ge 0 : \sum_{i=1}^{n} \int_{T_{K-1}}^{T} \upsilon_i(\tau) \,{\rm d}\tau \ge \left\lceil\frac{1}{4q}\right\rceil \right\}
\end{align}
for all $K\geq 1$.
\end{lem}

\begin{proof}
For each $r\in[d]$, we let $\tilde t_r$ be the (random) wall-clock time when the $r$th coordinate becomes revealed, i.e.,
\begin{align*}
\tilde t_r
:= \min\Bigl\{t\ge 0:\text{there exists}~k\in\fS_t~\text{such that}~\operatorname{prog}(x_k)\ge r\Bigr\}
\end{align*}
with the convention $\tilde t_0:=0$ and $\min\{\emptyset\} = +\infty$.

We define $\vartheta_r$ as the number of completed oracle outputs needed to reveal coordinate $r$ given that
coordinates $1,\dots,r-1$ have already been revealed. 
Recall that the function satisfies 
$\operatorname{prog}(\nabla f(x))\le \operatorname{prog}(x)+1$ for all $x\in\BR^d$, then revealing the new coordinate only occurs when we take the stochastic first-order oracle $\widehat{\nabla} F(x;\xi)$ defined in \eqref{eq:tyurin_mapping_212} with $\xi=1$.
Therefore, conditionally on $\vartheta_{1},\dots,\vartheta_{r-1}$, we have
\begin{align*}
\vartheta_r \sim \mathrm{Geometric}(q),
\end{align*}
since all stochastic first-order oracle calls are performed independently.

Next, we connect the above random variable $\vartheta_r$ to the time complexity under the universal computation model.
We define the sequence of (random) times $\{t_r\}_{r\geq 0}$ as
\begin{align}
\label{eq:def_tr_like_tyurin}
t_r
:= \begin{cases}
0, & r=0, \\
\min\left\{t\ge 0:\sum_{i=1}^{n}
\left\lfloor
V_i(t)-V_i(t_{r-1})
\right\rfloor
\ge
\vartheta_r
\right\}, & r \geq 1,
\end{cases}
\end{align}
where the minimum exists by Lemma~\ref{lem:tyurin25_G1}.
According to Assumption~\ref{ass:universalmodel},
in the interval $[\tilde t_{r-1},t]$,
all workers together can access at most
\begin{align*}
\sum_{i=1}^n \lfloor V_i(t)-V_i(\tilde t_{r-1})\rfloor    
\end{align*}
stochastic first-order oracles.
Hence, no algorithm can reveal coordinate $r$ before time $t_r$, i.e.,
\begin{align}
\label{eq:tilde_ge_t}
\tilde t_r \ge t_r,
\end{align}
for all $r\in[d]$.

Apply to Lemma~\ref{lem:tyurin25_G2} with $T=d$ and $\pi_{r,\vartheta_1,\dots,\vartheta_{r-1}}=q$ for all $r\in[d]$, then we have
\begin{align*}
\Pr\!\left(
\sum_{r=1}^{d}\mathbf{1}\!\left[\vartheta_r>\frac{1}{4q}\right]
\le \frac{d}{2}+\log \delta
\right)
\le \delta
\end{align*}
for any $\delta\in(0,1]$.
Therefore, with probability at least $1-\delta$, it holds
\begin{align}
\label{eq:many_large_vartheta_univ}
\sum_{r=1}^{d}\mathbf{1}\!\left[\vartheta_r>\frac{1}{4q}\right]
> \frac{d}{2}+\log\delta.
\end{align}

If $\left\lfloor d/2+\log\delta \right\rfloor\le 0$, then the conclusion is trivial since
\begin{align*}
T_{\left\lfloor {d}/{2}+\log\delta \right\rfloor}=T_0=0.    
\end{align*}
Hence, we focus on the case of $\left\lfloor {d}/{2}+\log\delta \right\rfloor\ge 1$.
Based on the expression on the left-hand side of \eqref{eq:many_large_vartheta_univ},
there exist indices $j_1,j_2,\dots,j_{\left\lfloor {d}/{2}+\log\delta \right\rfloor}$ such that $1\le j_1<j_2<\cdots<j_{\left\lfloor {d}/{2}+\log\delta \right\rfloor}\le d$ and
\begin{align}\label{eq:theta-4q}
\vartheta_{j_k}>\frac{1}{4q}
\end{align}
for all $k=1,\dots,\left\lfloor {d}/{2}+\log\delta \right\rfloor$.
Since the variable $\vartheta_{j_k}$ is an integer, it holds
\begin{align}
\label{eq:vartheta_ge_ceil}
\vartheta_{j_k}\ge\left\lceil\frac{1}{4q}\right\rceil.
%\qquad \forall k\in\Bigl[\left\lfloor \frac{d}{2}+\log\delta \right\rfloor\Bigr].
\end{align}

We now compare the times $\{t_r\}$ defined in \eqref{eq:def_tr_like_tyurin} with the sequence
$\{T_k\}$ defined in \eqref{eq:TK_universal_recursion_aligned}.
Specifically, we target to prove 
\begin{align}\label{eq:tjk_ge_Tk}
t_{j_k}\ge T_k
\end{align}
holds for all $k=1,\dots,\left\lfloor {d}/{2}+\log\delta \right\rfloor$ by induction.

\paragraph{The Induction Base ($k=1$).}
Since we require $j_1>1$, \eqref{eq:def_tr_like_tyurin} implies
\begin{align}\label{eq:tj1}
t_{j_1}
= \min\left\{t\ge 0:\sum_{i=1}^{n}
\left\lfloor
V_i(t)-V_i(t_{j_1-1})
\right\rfloor
\ge
\vartheta_{j_1}
\right\}.
\end{align}
Recall that $V_i(t_{j_1-1})\ge 0$ and each $V_i$ is non-decreasing, then for all $t\ge 0$, we have 
\begin{align*}
V_i(t)\ge V_i(t)-V_i(t_{j_1-1}),     
\end{align*}
that is
\begin{align*}
\left\lfloor V_i(t)\right\rfloor
\ge
\left\lfloor V_i(t)-V_i(t_{j_1-1})\right\rfloor.
\end{align*}
Therefore, at time $t=t_{j_1}$, it holds
\begin{align}\label{eq:V14q}
\sum_{i=1}^n \left\lfloor V_i(t_{j_1})\right\rfloor
\ge
\sum_{i=1}^n \left\lfloor V_i(t_{j_1})-V_i(t_{j_1-1})\right\rfloor
\ge
\vartheta_{j_1}
\ge
\left\lceil\frac{1}{4q}\right\rceil,
\end{align}
where the second step is based on \eqref{eq:def_tr_like_tyurin} and the last step is based on \eqref{eq:vartheta_ge_ceil}. 
Since \eqref{eq:TK_universal_recursion_aligned} defines~
$T_0=0$ and $T_1=\min\{t\ge 0:\sum_{i=1}^n \lfloor V_i(t)-V_i(T_0)\rfloor \ge \lceil 1/(4q)\rceil\}$, we have $t_{j_1}\ge T_1$ based on \eqref{eq:V14q} and definition (\ref{eq:tj1}).

\paragraph{Induction Step $(k\geq 2)$.}
We assume $t_{j_{k-1}}\ge T_{k-1}$ for some $k\geq 2$.
Recall that our algorithm class means revealing coordinate $j_k$ requires revealing
coordinate $j_k-1$ at first, then we have
\begin{align*}
t_{j_k-1}\ge t_{j_{k-1}}\ge T_{k-1}.    
\end{align*}
By monotonicity of each $V_i$, for all $i\in[n]$, it holds
\begin{align*}
V_i(t_{j_k-1}) \ge V_i(T_{k-1}).
\end{align*}
Hence, for all $t\ge 0$, we have
\begin{align*}
V_i(t)-V_i(T_{k-1})
\ge
V_i(t)-V_i(t_{j_k-1}),
\end{align*}
that is
\begin{align*}
\left\lfloor V_i(t)-V_i(T_{k-1})\right\rfloor
\ge
\left\lfloor V_i(t)-V_i(t_{j_k-1})\right\rfloor.
\end{align*}
For the time $t=t_{j_k}$, the above inequality and \eqref{eq:def_tr_like_tyurin} gives
\begin{align*}
\sum_{i=1}^n \left\lfloor V_i(t_{j_k})-V_i(T_{k-1})\right\rfloor
\ge
\sum_{i=1}^n \left\lfloor V_i(t_{j_k})-V_i(t_{j_k-1})\right\rfloor
\ge
\vartheta_{j_k}
\ge
\left\lceil\frac{1}{4q}\right\rceil,
\end{align*}
where the last step again uses \eqref{eq:vartheta_ge_ceil}.
By the definition of $T_k$ in \eqref{eq:TK_universal_recursion_aligned}, we have $t_{j_k}\ge T_k$.
Hence, we complete the induction to finish the proof of  \eqref{eq:tjk_ge_Tk}.

Finally, since $j_{\left\lfloor {d}/{2}+\log\delta \right\rfloor}\le d$ and revealing coordinate $d$ requires revealing
coordinate $j_{\left\lfloor {d}/{2}+\log\delta \right\rfloor}$ first, we have
\begin{align*}
\tilde t_d \ge \tilde t_{j_{\left\lfloor {d}/{2}+\log\delta \right\rfloor}}
\overset{(\ref{eq:tilde_ge_t})}{\ge}
t_{j_{\left\lfloor {d}/{2}+\log\delta \right\rfloor}}
\overset{(\ref{eq:tjk_ge_Tk})}{\ge}
T_{\left\lfloor {d}/{2}+\log\delta \right\rfloor}.
\end{align*}
Equivalently, for all $t\le T_{\left\lfloor {d}/{2}+\log\delta \right\rfloor}$,
no available iterate can have $\operatorname{prog}(x_k)\ge d$, i.e., it holds 
\begin{align*}
\inf_{k\in \fS_t}\mathbf{1}\!\left[\operatorname{prog}(x_k)<d\right]\ge 1,
\end{align*}
for all $t\le T_{\left\lfloor {d}/{2}+\log\delta \right\rfloor}$ with probability at least $1-\delta$.
\end{proof}

Finally, we prove Theorem \ref{thm:lb_asgd_pBCM_universal} as follows.
\begin{proof}
We follow the proof of Theorem~\ref{thm:lb_asgd_pBCM} in Appendix \ref{appendix:lower-fixed} to construct the function and stochastic first-order oracle, i.e., we define
\begin{align*}
f(x):=\frac{L\lambda^2}{152}H_d\left(\frac{x}{\lambda}\right),
\qquad
\lambda:=\frac{608\epsilon}{L},
\qquad \text{and} \qquad
d:=\left\lfloor\frac{L\Delta}{29184\epsilon^{2}}\right\rfloor;
\end{align*}
the stochastic first-order oracle  $\widehat{\nabla} F$ is defined as
\begin{equation*}
\left(\widehat{\nabla} F(x;\xi)\right)_{(j)}
\ =\
\left(\nabla f(x)\left(1+\mathbf{1}\!\left[j>\operatorname{prog}(x)\right]\left(\frac{\xi}{q}-1\right)\right)\right)_{(j)},
\end{equation*}
for all $x\in\R^{d}$, $\xi\in\{0,1\}$, and $j\in[d]$, where we take $\xi \sim \mathrm{Bernoulli}(q)$ for some $q\in(0,1]$.
We then consider the algorithm class in Definition \ref{dfn:time_multi_oracle_universal} associated with above $f$ and $\widehat{\nabla} F$.

We then complete the proof by considering the function class, the oracle class, and the time complexity.
\paragraph{The Function Class and the Oracle Class.}

Following the verification in the proof of Theorem \ref{thm:lb_asgd_pBCM} in Appendix \ref{appendix:lower-fixed}, we know that that $f$ is $L$-smooth with $f(0)-\inf_x f(x)\le \Delta$, and that the stochastic first-order oracle $\widehat{\nabla} F$ satisfies the desired $p$-BCM condition with $\sigma>0$  and $p\in(1,2]$.

\paragraph{The Time Complexity.}

Note that the derivation of \eqref{eq:lb_bridge_explicit_final} in Appendix \ref{appendix:lower-fixed} only depends on the constructions of $f$ the objective and stochastic first-order oracle $\widehat{\nabla} F$, so that it also works for the universal computation model, i.e., we have
\begin{equation}\label{eq:state-50}
\inf_{k\in \fS_t}\Norm{\nabla f(x_k)}
\ \ge\
4\eps\cdot \inf_{k\in \fS_t}\mathbf{1} \left[\operatorname{prog}(x_k)<d\right],
\end{equation}
where $\fS_t=\{k:\ t_k\le t\}$ collects all indices on the server whose query times do not exceed $t$.
Applying Lemma~\ref{lem:async_progress_universal} with $\delta=1/2$, we know that with probability at least $1/2$, it holds
\begin{align}\label{eq:final-12}
\inf_{k\in \fS_t}\mathbf{1} \left[\operatorname{prog}(x_k)<d\right]\ge 1
\end{align}
for all $t\le T_{\left\lfloor d/2+\log(1/2)\right\rfloor}$. 

Hence, for all $t\le T_{\left\lfloor d/2+\log(1/2)\right\rfloor}$, we have
\begin{align*}
\E\left[\inf_{k\in \fS_t}\Norm{\nabla f(x_k)}\right]
\ge
4\epsilon\cdot \Pr\!\left(\inf_{k\in \fS_t}\mathbf{1}\!\left[\operatorname{prog}(x_k)<d\right]\ge 1\right)
\ge
2\epsilon
> \epsilon,
\end{align*}
where the first step is based on \eqref{eq:state-50} and the second step is due to \eqref{eq:final-12} holds with probability at least $1/2$.
Consequently, if 
\begin{align*}
\E[\inf_{k\in \fS_t}\|\nabla f(x_k)\|]\le \epsilon
\end{align*}
holds for some $t$, then it is necessary that
\begin{align*}
t> T_{\left\lfloor d/2+\log(1/2)\right\rfloor}.    
\end{align*}
Combining with the definition of $T_K$ and the setting of $d$, we finish the proof.
\end{proof}

\end{document}

%% file: math.tex
\usepackage{amsmath}\allowdisplaybreaks
\usepackage{amsfonts,bm}
\usepackage{amssymb}

\def\eqref#1{equation~(\ref{#1})}
\def\floor#1{\left\lfloor #1 \right\rfloor}
\def\1{\bf{1}}

\newcommand{\Norm}[1]{\left\| #1 \right\|}

\def\inner#1#2{\left\langle #1, #2 \right\rangle}

\def\eps{{\epsilon}}

\def\vone{{\bf{1}}}

\def\fA{{\mathcal{A}}}

\def\fD{{\mathcal{D}}}

\def\fO{{\mathcal{O}}}

\def\fS{{\mathcal{S}}}

\def\BE{{\mathbb{E}}}
\def\BN{{\mathbb{N}}}

\def\BR{{\mathbb{R}}}

\newcommand{\E}{\mathbb{E}}

\newcommand{\R}{\mathbb{R}}

\usepackage{amsthm}
\theoremstyle{plain}
\newtheorem{thm}{Theorem}%[section]
\newtheorem{dfn}{Definition}

\newtheorem{lem}{Lemma}
\newtheorem{asm}{Assumption}
\newtheorem{remark}{Remark}
\newtheorem{cor}{Corollary}

\makeatletter
\def\Ddots{\mathinner{\mkern1mu\raise\p@
\vbox{\kern7\p@\hbox{.}}\mkern2mu
\raise4\p@\hbox{.}\mkern2mu\raise7\p@\hbox{.}\mkern1mu}}
\makeatother

\makeatletter
\newcommand*{\rom}[1]{\expandafter\@slowromancap\romannumeral #1@}
\makeatother